\documentclass[a4paper,12pt]{amsart}
\usepackage{amsmath,amsfonts,amssymb,graphics,mathrsfs,amsthm,eurosym,verbatim,quotes,graphicx}

\usepackage[shortlabels]{enumitem}

\usepackage[marginratio=1:1,tmargin=117pt, height=650pt]{geometry}

\usepackage{hyperref}
\usepackage[usenames, dvipsnames]{color}
\usepackage[flushmargin]{footmisc}

\usepackage[noadjust]{cite}

\numberwithin{equation}{section}
\theoremstyle{plain}
\newtheorem{theorem}{Theorem}[section]
\newtheorem{proposition}[theorem]{Proposition}
\newtheorem{cor}[theorem]{Corollary}
\newtheorem{lemma}[theorem]{Lemma}

\theoremstyle{definition}

\newtheorem{example}{Example}[section]

\theoremstyle{definition} 
\newtheorem*{remark}{Remark}
\newtheorem*{remarks}{Remarks}

\newcommand*{\defeq}{\mathrel{\vcenter{\baselineskip0.5ex \lineskiplimit0pt
                     \hbox{\scriptsize.}\hbox{\scriptsize.}}}%
                     =}

\usepackage{subfig}
 \newtheoremstyle{claimstyle}%
   {}
   {}
   {\normalfont}
   {}
   {\itshape}
   {.}
   { }
   {\thmnote{#3}}

\theoremstyle{claimstyle}
\newtheorem*{varclaim}{}

\newenvironment{claim}[1][Claim]{\begin{varclaim}[#1]}{\end{varclaim}}

\newenvironment{subproof}{\begin{proof}}{%
               \end{proof}}

\newcommand{\C}{{\mathbb{C}}}
\newcommand{\B}{\mathcal B}
\newcommand{\D}{{\mathbb{D}}}
\newcommand{\HH}{{\mathbb{H}}}
\newcommand{\R}{{\mathbb{R}}}
\newcommand{\Z}{{\mathbb{Z}}}
\newcommand{\N}{{\mathbb{N}}}

\renewcommand{\Re}{\operatorname{Re}}
\renewcommand{\Im}{\operatorname{Im}}
\newcommand{\im}{\operatorname{Im}}

\newcommand{\dist}{\operatorname{dist}}
\newcommand{\tef}{transcendental entire function}
\newcommand{\qfor}{\quad\text{for }}
\begin{document}
\title[Fatou's associates]{Fatou's associates}

\dedicatory{For Misha Lyubich on the occasion of his 60th birthday}

\author[V.~Evdoridou]{Vasiliki Evdoridou}
\address{School of Mathematics and Statistics\\ The Open University\\
Milton Keynes MK7 6AA\\ UK \textsc{\newline\indent\href{https://orcid.org/0000-0002-5409-2663}{\includegraphics[width=1em,height=1em]{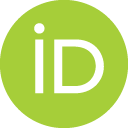} {\normalfont https://orcid.org/0000-0002-5409-2663}}}}
\email{vasiliki.evdoridou@open.ac.uk}

	 \author[L.~Rempe]{Lasse Rempe}
\address{Dept. of Mathematical Sciences \\
	 University of Liverpool \\
   Liverpool L69 7ZL\\
   UK 
	 \textsc{\newline\indent\href{https://orcid.org/0000-0001-8032-8580}{\includegraphics[width=1em,height=1em]{orcid2.png} {\normalfont https://orcid.org/0000-0001-8032-8580}}}}
\email{l.rempe@liverpool.ac.uk}
	 
\author[D.~J.~Sixsmith]{David J.~Sixsmith}
\address{Department of Mathematical Sciences \\
	 University of Liverpool \\
   Liverpool L69 7ZL
   UK \newline \indent \href{https://orcid.org/0000-0002-3543-6969}{\includegraphics[width=1em,height=1em]{orcid2.png} {\normalfont https://orcid.org/0000-0002-3543-6969}}} 
\email{david.sixsmith@open.ac.uk}
\thanks{The first author was supported by Engineering and Physical Sciences Research Council grant EP/R010560/1}
\subjclass[2020]{Primary 37F10; Secondary 30D05, 30J05}
\keywords{holomorphic dynamics, inner functions}

\begin{abstract}
Suppose that $f$ is a transcendental entire function,  $V \subsetneq \C$ is a simply-connected domain, and $U$ is a connected component of $f^{-1}(V)$. Using Riemann maps, we associate the map $f \colon U \to V$ to an inner function $g \colon \D \to \D$. It is straightforward to see that $g$ is either a finite Blaschke product, or, with an appropriate normalisation, can be taken to be an infinite Blaschke product. 

We show that when the singular values of $f$ in $V$ lie away from the boundary, there is a strong relationship between singularities of $g$ and accesses to infinity in $U$. In the case where $U$ is a forward-invariant Fatou component of $f$, this leads to a very significant generalisation of earlier results on the number of singularities of the map $g$. 

If $U$ is a forward-invariant Fatou component of $f$ there are currently very few examples where the relationship between the pair $(f, U)$ and the function $g$ has been calculated. We study this relationship for several well-known families of transcendental entire functions. 

It is also natural to ask which finite Blaschke products can arise in this manner, and we show the following: For every finite Blaschke product $g$ whose Julia
set coincides with the unit circle, 
there exists a transcendental entire function $f$ with an invariant Fatou
 component  such that $g$ is associated to $f$ in the above sense. Furthermore,
  there exists a single transcendental entire function $f$ with the property that any 
  finite Blaschke product can be arbitrarily closely approximated by an inner 
  function associated to the restriction of $f$ to a wandering domain. 
\end{abstract}
\maketitle
\section{Introduction}
Although much of this paper concerns dynamics, we begin with the following simple and very general non-dynamical construction.
Suppose that $f$ is a {\tef}, that $V \subsetneq \C$ is a simply-connected domain, and that $U$ is a connected component of $f^{-1}(V)$; it is a consequence of the Open Mapping Theorem that $U$ is also simply connected. Let  $\phi \colon \D \to U$ and $\psi \colon \D \to V$ be Riemann maps, and set $g \defeq \psi^{-1} \circ f \circ \phi$; see Figure~\ref{fig:inner}. We begin with a  result that summarises the properties of the map $g$. This is not entirely new but to the best of our knowledge it has not been stated in this generality before. Here an \emph{inner function} is a holomorphic self-map of $\D$ for which radial limits exist at almost all points of the unit circle, and belong to the unit circle. A particular class of inner functions is the class of \emph{Blaschke products}\footnote{Sometimes, the term ``Blaschke product'' is used more generally for a function of the form~\eqref{eq:Bdef} where some $a_n$ may also have $\lvert a_n\rvert> 1$, so that $B$ has poles in $\D$. These are not inner functions, and are not considered in this paper.}. These are functions of the form
\begin{equation}
\label{eq:Bdef}
B(z) \defeq e^{i\theta} \prod_{n = 1}^d \frac{|a_n|}{a_n} \frac{a_n-z}{1-\overline{a_n}z},
\end{equation}
where $\theta \in \R$, $d \in \N \cup \{\infty\}$, and $(a_n)_{1 \leq n \leq d}$ is a sequence of points of $\D$, which satisfies the condition $\sum (1-|a_n|) < \infty$. When $a_n = 0$ we interpret the term in the infinite product simply as $z$. If $d$ is finite then $B$ is called a \emph{finite Blaschke product of degree $d$}, and otherwise it is an \emph{infinite Blaschke product}.

\begin{proposition}
\label{prop:basics}
Suppose that $f$ is a {\tef}, that $V \subsetneq \C$ is a simply-connected domain, and that $U$ is a connected component of $f^{-1}(V)$. Let $\phi \colon \D \to U$ and $\psi \colon \D \to V$ be conformal, and set $g \defeq \psi^{-1} \circ f \circ \phi$. Then $g$ is an inner function, which, for an appropriate choice of $\phi$ and $\psi$, can be taken to be a Blaschke product. More precisely, exactly one of the following conditions holds.
\begin{enumerate}[(a)]
\item \emph{Finite valence:} $g$ is a finite Blaschke product of degree $d$, for some $d \in \N$, and $f|_U$ is of constant finite valence $d$.\label{theo:a}
\item \emph{Infinite valence:} $g$ is an infinite Blaschke product, and $U \cap f^{-1}(z)$ is infinite for all $z \in V$ with at most one exception.\label{theo:b}
\end{enumerate}
\end{proposition}
\begin{figure}
\begin{center}
\def\svgwidth{.7\textwidth}
\input{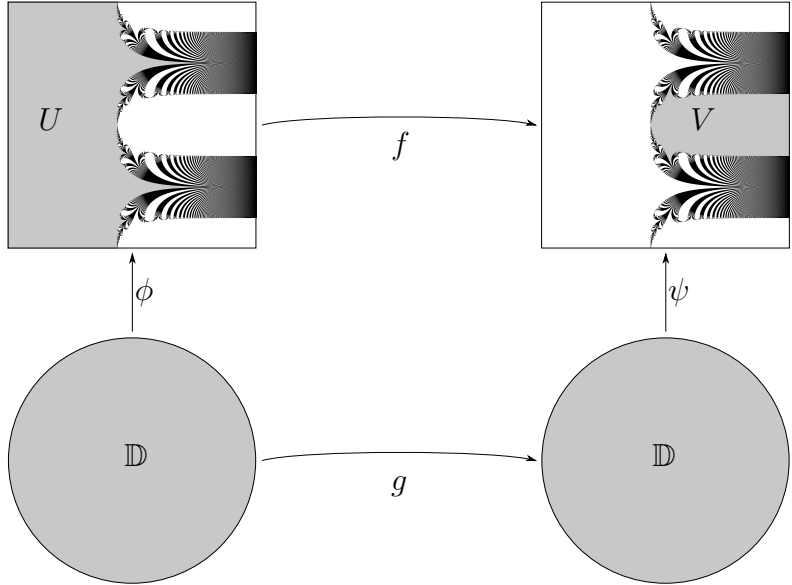}
\end{center}
\caption{\label{fig:inner}Construction of an associated inner function. Here, $f(z)=\lambda e^z$ is an exponential function 
with real $\lambda<-e$, which has 
an attracting periodic orbit of period $2$. The domains $U$ and $V$ shown are the two periodic Fatou components of $f$.}
\end{figure}
In the setting of Proposition~\ref{prop:basics}, we say that $g$ is an \emph{inner function associated to} $f \lvert _{U}$. 
 Such inner functions have been mostly considered in a dynamical setting where $U=V$ and $\phi=\psi$, see below. However,
 they have also appeared in settings where $U\neq V$; see, for example, \cite[p.5]{Bishop}.
\begin{remark}
Note that Proposition~\ref{prop:basics} implies that there are many inner functions which cannot be associated to a transcendental entire function in the sense of this paper. For example, if $A$ is any closed subset of $\D$, of (logarithmic) capacity zero, then there is an inner function that omits all the points of $A$; see \cite{MR0481015}.
\end{remark}

In our first main result, which significantly generalises earlier results in a dynamical setting, we are interested in the singularities of the associated inner function; a point $\zeta \in \partial \mathbb{D}$ is called a \emph{singularity} of (an inner function) $g$, if $g$ cannot be extended holomorphically to any neighbourhood of $\zeta$ in $\mathbb{C}.$ For a transcendental entire function $f$, we denote by $S(f)$ the set of \emph{singular values} of $f$; in other words, the closure of the set of critical and finite asymptotic values of $f$. Our result is as follows.
\begin{theorem}
\label{theo:tracts}
Suppose that $f$ is a {\tef},  that $V \subsetneq \C$ is a simply-connected domain, and that 
$U$ is a component of $f^{-1}(V)$ such that $f\colon U\to V$ is of infinite valence.
Suppose that $S(f)\cap V$ is compact, and let $D$ be 
 a bounded Jordan domain with $S(f)\cap V\subset D$ and $\overline{D} \subset V$. 
 Then the singularities of an associated inner function $g$ 
  are in order-preserving one-to-one correspondence with the accesses to infinity in 
  $U \cap f^{-1}(D)$. In particular, the number of singularities of $g$ is equal to the number of components of $U \setminus f^{-1}(\overline{D})$. 
\end{theorem}
\begin{remark}\mbox{}
\begin{enumerate}[(a)]
\item An \emph{access to infinity} in $U\cap f^{-1}(D)$ is a homotopy class
  of curves to infinity in $U$; see Section~\ref{S.singularities}
\item By the final statement, we mean that the number of singularities and
  the number of components are either both infinite, or both finite and equal. 
  We caution that, when infinite, the number of singularities may be
  uncountable, while the number of components of 
  $U\setminus f^{-1}(\overline{D})$ is always countable. 
\item In the case of finite valence, it follows from Proposition~\ref{prop:basics} that any associated inner product is a finite Blaschke product, and has no singularities.
\end{enumerate}
\end{remark}

We now consider associated inner functions in a dynamical setting. Let $f$ be a transcendental entire function, and denote by $f^n$ the $n$th iterate of $f$. The set of points for which the set of iterates $\{f^n\}_{n \in \mathbb{N}}$ form a normal family in some neighbourhood is the \emph{Fatou set} $F(f)$, and its complement in the complex plane is the \emph{Julia set} $J(f)$. The Fatou set is open, and so consists of connected components which are called \emph{Fatou components}. For an introduction to the properties of these sets see, for example, \cite{bergweiler93}.

In the case that $U$ is a simply-connected Fatou component, and $V$ is the Fatou component containing $f(U)$, then the conditions we discussed earlier all hold, and we can associate an inner function to $f|_U$. 
A case of particular interest is when the Fatou component $U$ is \emph{forward invariant}, in other words such that $f(U) \subset U$. Note that it is well known that forward-invariant Fatou components are necessarily simply connected. In this case we have that $U = V$, we can set $\psi = \phi$, and the dynamics of $f$ on $U$ is \emph{conjugate} to the dynamics on $\D$ of the function $g \defeq \phi^{-1} \circ f \circ \phi$. Moreover, $g$ is unique in this respect, up to a conformal conjugacy. In this case we say that $g$ is an inner function \textit{dynamically associated to} $f|_U$. This construction appears to have first been considered
by T\"opfer~\cite[\S II]{topfer} in 1939. Its connection to inner functions appears to have first been
made by Kisaka~\cite{kisaka} and Baker and Dom\'inguez~\cite{BakerDominguez}. Compare also \cite{DevandG, fagella-henriksen, Baranski, BK,  univalentbd, Bargmann, EFJS}. 

Theorem~\ref{theo:tracts} has the following corollary, which is a significant generalisation of the main result of \cite{EFJS}. Here we consider the class $\B$ of {\tef}s for which $S(f)$ is bounded, and for a function $f \in \B$ a \emph{tract} is a component of $f^{-1}(\C \setminus \overline{D'})$ where $D'$ is a bounded Jordan domain containing $S(f)$. It is well-known that the number of tracts is independent of the choice of $D'$.
\begin{cor}
\label{corr:tracts}
Suppose that $f \in \B$, and that ${S(f)} \subset F(f)$. Suppose also that $U$ is an unbounded forward-invariant Fatou component of $f$. Then the number of singularities of a dynamically associated inner function is at most equal to the number of tracts of $f$.
\end{cor}
This generalises \cite[Theorem 1.5]{EFJS}, in which the condition that $S(f)$ is a compact subset of the Fatou set was replaced by the condition that the \emph{postsingular set} defined by 
\[
\mathcal{P}(f) \defeq \overline{\bigcup_{j \geq 0} f^j(S(f))},
\]
is a compact subset of the Fatou set; such functions are called \emph{hyperbolic}. 

Lyubich (personal communication) has asked which inner functions $g$ arise as dynamically associated inner functions.  Few authors have explicitly calculated $g$ for given functions
 $f$. Indeed, we are aware of only three examples in the literature. First, T\"opfer \cite[{\S}V]{topfer} 
  considered the function $f(z)=\sin(z)$, 
  which has a triple fixed point at the origin, with two attracting directions, along the positive and negative real axis. Note that,
  by symmetry, the same inner function $g$ can be dynamically associated to either immediate parabolic basin. 
  T\"opfer observes~\cite[p.~78]{topfer}
  that $g$ can be taken to have the form
     \begin{equation}\label{eqn:topfer} g(z) = \frac{z^2 + k}{kz^2+1} \end{equation}
  for some $0<k<1$. He does not determine the correct value of $k$, which is $k=1/3$ since the function $g$ must have
  a parabolic point at $z=1$ (see the remarks at the end of Section~\ref{sec:expo}). 
  
  Devaney and Goldberg \cite{DevandG} considered the Julia set of $f_\lambda(z) \defeq \lambda e^z$ for values of $\lambda$ such that $f$ has a completely invariant attracting basin $U$. They showed that these functions have a dynamically associated inner function of the form
\begin{equation}
\label{eq:DevandG}
g(z) = g_{\mu}(z) \defeq \exp\left(i \ \frac{\mu + \overline{\mu}z}{1+z}\right),
\end{equation}
where $\mu$ lies in the upper half-plane $\HH$, and depends on $\lambda$. Note that $g$ is not an infinite Blaschke product -- indeed, the proof technique used in \cite{DevandG} depends on this fact -- but is conjugate to one. However,
Devaney and Goldberg did not determine which values of $\mu$ are realised.

The result of Devaney and Goldberg was generalised later by Schleicher. He considered the case that $f_\lambda$ has an attracting \emph{periodic} point; in this case $f_{\lambda}$ is \emph{hyperbolic}. He observes that the associated 
inner functions can always be chosen to take a certain form, which is equivalent to~\eqref{eq:DevandG}; see \cite[Lemmas III 4.2, III 4.3]{Dierkhab} for details. 

Finally, Baker and Dom\'inguez \cite{BakerDominguez} showed that the map $f(z) \defeq z + e^{-z}$ has an invariant Fatou component with the dynamically associated Blaschke product
\begin{equation}\label{eqn:parabolicblaschke}
g(z) \defeq \frac{3z^2 + 1}{3 + z^2},
\end{equation}
which is the same as T\"opfer's map~\eqref{eqn:topfer}. It is easy to see how to write $g$ in the form \eqref{eq:Bdef}. 

In view of this, perhaps unexpected, dearth of specific examples, our next goal in this paper is to find classes of {\tef}s, $\mathcal{F}$, and classes of inner functions, $\mathcal{G}$, with the following properties;
\begin{enumerate}[(I)]
\item Each $f \in \mathcal{F}$ has a forward-invariant Fatou component $U$;\label{p1}
\item For each $f \in \mathcal{F}$ there is an inner function $g\in\mathcal{G}$ dynamically associated to~$f$;\label{p2}
\item For each $g \in \mathcal{G}$ there is an $f \in \mathcal{F}$ such that $g$ is dynamically associated to~$f$.\label{p3}
\end{enumerate}

We begin with a result of this form for finite Blaschke products. If such a Blaschke product 
    $g\colon\D\to\D$ is dynamically associated to an invariant attracting or parabolic Fatou component
    of an entire function, then $g$ has either an attracting fixed point in 
    $\D$, or a triple fixed point on $S^1=\partial \D$. Equivalently,
      $J(g)=S^1$; see~\cite{fletcherblaschke}. 
    
 \begin{theorem}\label{theo:finiteblaschke}
   Let $\mathcal{F}$ consist of all entire functions having a forward-invariant 
   attracting or parabolic Fatou component
      $U$ that is a bounded Jordan domain. Let 
      $\mathcal{G}$ be the connectedness locus of finite Blaschke products;
       that is, $\mathcal{G}$ consists of all finite Blaschke products of degree $d\geq 2$
        with $J(g)= S^1$. 
     Then~\ref{p1}, \ref{p2} and \ref{p3} hold. 
  \end{theorem}
  Observe that~\ref{p1} and~\ref{p2} hold by assumption and Proposition~\ref{prop:basics},
    so the main content of the theorem is showing the existence of
    an entire function realising a prescribed Blaschke product as its dynamically
    associated inner function. This will be achieved by quasiconformal surgery. 

Next, we give
 a complete description of dynamical inner functions for exponential maps with attracting fixed points, thus completing 
  the work of Devaney and Goldberg. 
 Similarly as in \cite{Dierkhab}, we find it convenient to change coordinates from the unit disc to the upper half-plane and consider \textit{inner functions of the upper half-plane} 
  associated to $f|_U$; compare also \cite{Bargmann}. We use the
   family of functions
   \begin{equation}\label{eqn:tangent} g_{a,b}\colon \HH\to\HH \defeq  a\tan(z) + b, \qquad a>0, b\in  (-\pi/2,\pi/2]. \end{equation}
  It is easy to check that $g_{a,b}$ is conjugate to $g_{\mu}$ as in~\eqref{eq:DevandG} for $\mu=2(b+ai)$. Thus the following result
   also gives a complete description of the set of $\mu$ for which $g_{\mu}$ 
   arises as an inner function of an exponential map. (Compare Figure~\ref{fig:tan_inner} in Section~\ref{sec:expo}.)
   \begin{theorem}\label{theo:exp}
  Set
     \[ \mathcal{F} \defeq \{f_{\lambda}\colon f_{\lambda} \text{ has 
                an attracting fixed point} \} =
                 \{ f_{\tau\cdot e^{-\tau}}\colon \tau\in\D\setminus\{0\} \}, \]
      where $f_{\lambda}(z)=\lambda e^z$. Also let 
    \begin{align} \label{eqn:expoG} \mathcal{G} \defeq& \{ g_{a,b}\colon g_{a,b}\text{ has an attracting fixed point in $\D$} \}  \\
        =& \left\{ g_{a,b}\colon a>1 \text{ or } \lvert b\rvert > \arccos(\sqrt{a}) - \sqrt{a}\cdot \sqrt{1-a} \right\}.\notag \end{align}
  Then $\mathcal{F}$ and $\mathcal{G}$ satisfy~\ref{p1}, \ref{p2} and \ref{p3}. More precisely, for every $f_{\lambda}\in\mathcal{F}$, the family
    $\mathcal{G}$ contains exactly one 
   dynamically associated inner function of $f$, and vice versa. 
\end{theorem}

As in Schleicher's work, our theorem applies also to iterates of exponential maps. Indeed, more generally the following is true. 
\begin{theorem}\label{theo:unisingular}
  Let $\mathcal{G}$ be as above, 
     let $f$ be a transcendental entire function, and suppose that $U$ is a basin of attraction of period $n$ for $f$. If the cycle of $U$ contains only one singular value of $f$, and 
   $f^n\colon U\to U$ is of infinite valence, then the family $\mathcal{G}$ from~\eqref{eqn:expoG} contains exactly one 
    inner function dynamically associated to $f^n|_U$. 
\end{theorem} 

 While \cite{DevandG} and \cite[Section~III.4]{Dierkhab} only treated
   attracting dynamics, we can also consider parabolic basins, 
    thus completing the description of associated inner functions for periodic 
    Fatou components of exponential maps. 

 \begin{theorem}\label{theo:unisingularparabolic}
  Let $f$ be a transcendental entire function, and suppose that $U$ is a parabolic basin of period $n$ for $f$. If the cycle of $U$ contains only one singular value of $f$, and 
   $f^n\colon U\to U$ is of infinite valence, then $\tan\colon \HH\to\HH$ is a dynamically associated inner function for $f^n$ on $U$. 
\end{theorem} 

 We are able to give a similar description of dynamically associated inner functions in cases where $f^n\colon U\to U$ takes some value only finitely many times, 
   see Theorem~\ref{theo:finite}. In particular, this applies to many functions of the form 
\[
f(z) \defeq \lambda P(z) e^{Q(z)};
\]
  see Corollary~\ref{cor:exponentials}.

The case of Fatou components containing infinitely many critical points is more complicated. We begin with the following detailed example, which concerns sine functions with invariant Fatou components of infinite valence. (Recall that the example $f(z)=\sin(z)$ studied
by T\"opfer has invariant Fatou components of finite valence.) 
\begin{theorem}
\label{theo.sine}
There is a homeomorphism $\psi \colon (0, 1) \to (1, \infty)$ with the following property. Let $\mathcal{F}$ be the family of {\tef}s
\begin{equation}
\label{sinfdef}
f_\lambda(z) \defeq \lambda \sin z, \qfor \lambda \in (0, 1).
\end{equation}
For $\tau > 1$ let
\[
a_n = a_n(\tau) \defeq \frac{\tau^n-1}{\tau^n+1}, \qfor n \in \N,
\]
and let $\mathcal{G}$ be the family of infinite Blaschke products
\begin{equation}
\label{singdef}
g_\tau(z) \defeq z \prod_{n=1}^\infty \frac{a_n^2 - z^2}{1-a_n^2z^2}.
\end{equation}
Then \ref{p1}, \ref{p2} and \ref{p3} hold for these families, with $\tau = \psi(\lambda)$.
\end{theorem}
\begin{remark}\normalfont
The proof of Theorem~\ref{theo.sine} makes strong use of symmetries of the Julia sets of the functions involved. It is not easy to see, therefore, how one might extend these results to wider families. 
\end{remark}

Let $g$ be an inner function dynamically associated to an invariant Fatou component $U$ of infinite valence for a transcendental entire function $f$. 
  Let $A$ denote the set of singularities of $g$. By the
  Schwarz reflection principle, $g$ extends to a meromorphic function on $\hat{\C}\setminus A$, so we can think of $g$ as a global complex dynamical system. 
  As mentioned above, if $U$ is an attracting or parabolic basin, we have $J(g)=S^1$; see \cite[Lemma~2]{kisaka}, 
  \cite[Lemmas~8 and~9]{BakerDominguez} and \cite[Theorem~2.24]{Bargmann}. Here 
  a point on the unit circle is in 
  the Julia set $J(g)$ if it has no neighbourhood on which the iterates
 of $g$ are defined and normal \cite[Section~3]{BakerDominguez}. If $\# A=1$, then $g$ is (up to conformal conjugacy) a transcendental meromorphic function. When $A$ is countable,
   $g$ belongs to a class of functions for which the theory of complex dynamics was developed by Bolsch~\cite{bolsch}. Similarly, if $f$ has only finitely many singular
 values in the Fatou component $U$, then $g$ is a \emph{finite-type map} in the sense of Epstein~\cite{epsteinthesis}; see also~\cite{remperadial,epsteincheritat}. In particular, inner functions
  allow us to construct many examples of functions in these classes for which the Julia set is a circle. A larger class of holomorphic functions for which Fatou-Julia iteration theory
  has been extended is Epstein's theory of \emph{Ahlfors islands maps}; see \cite{remperadial,remperipponexotic}, and also~\cite{bakerdominguezherring} for the case
   where $A\neq S^1$. It would be interesting to investigate when the inner function $g$ associated to $f$ satisfies this Ahlfors islands condition.

 For families of entire functions with a finite number of singular values, it is plausible that the preceding observation about finite-type maps, together with surgery
  techniques similar to Theorem~\ref{theo:finiteblaschke}, can lead to a 
   description of the associated inner functions. On the other hand, it appears to be very difficult
  to develop general principles for inner functions associated to Fatou components where the singular values are allowed to accumulate on the boundary. It is perhaps 
  surprising that we can nonetheless give a very precise description in one particular case. 
\begin{theorem}
\label{theo.Fatou}
Let $\mathcal{F}$ be the family of {\tef}s
\begin{equation}
\label{Fatoudef}
f_\lambda(z) \defeq \lambda + z + e^{-z}, \qfor \lambda > 0.
\end{equation}
Let $\mathcal{G}$ be the family of maps 
\begin{equation}
\label{eq:gmudef}
g_{\lambda} \colon \HH \to \HH; \quad g_{\lambda}(z) \defeq z - \lambda \frac{\cot z}{2}, \qfor \lambda > 0.
\end{equation}
Then \ref{p1}, \ref{p2} and \ref{p3} hold for these families, with $g_{\lambda}$ being associated to the restriction of $f_{\lambda}$ to its single Fatou component.
\end{theorem}
\begin{remarks}
\mbox{ }
\begin{enumerate}
\item Again, the proof of this result makes strong use of symmetries of Julia sets, and does not extend to the case
 where $\lambda$ is not real and positive. 
\item The map $g_\lambda$ is conjugate to the map $h_\lambda(z) \defeq z + \lambda \tan z$ via the conjugation $z \mapsto \pi/2 - z$. We prefer the parameterisation $g_\lambda$ as it makes the proof slightly simpler. Note that the dynamics of the map $h_1$ was studied in \cite{Accesses}.
\item The dynamics of the maps $f_{\lambda}$ was studied in \cite{Fweb}, under the parameterisation 
\[
    h(z) \defeq z + a + be^{cz},  \quad b \ne 0, \ ac < 0.
\]
(Up to affine conjugacy, this is the same family as $\mathcal{F}$.) 
\end{enumerate}
\end{remarks}

 To conclude, let us return to the case of inner functions associated to $f\colon U\to V$ where $U\neq V$, and in particular to the  case where $U$ and $V$ are 
  simply-connected \emph{wandering domains} of a transcendental entire function $f$. Similarly as in Theorem~\ref{theo:finiteblaschke}, we show that
  every finite Blaschke product may arise in this manner.

\begin{theorem}\label{theo:wandering}
 Let $g\colon \D\to\D$ be a Blaschke product of degree $d$, with $2\leq d<\infty$. Then there is a transcendental entire function $f$ having wandering 
   domains $U$ and $V=f(U)$ which are bounded Jordan domains and such that $g$ is an inner function associated to $f|_U$. 
\end{theorem}

Moreover, using approximation theory we can construct a \emph{single} entire function $f$ with an orbit of
wandering domains whose associated inner functions approximate any desired 
  Blaschke product: 

\begin{theorem}
\label{theo:finitecomplicated}
There is a {\tef} $f$ having a simply-connected wandering domain $U$ with the following property. Given a finite Blaschke product $B$ and $\epsilon > 0$, there is $n\geq 0$ and
 a Blaschke product $g$ associated to $f\colon f^n(U)\to f^{n+1}(U)$, such that the following both hold.
\begin{enumerate}[(i)]
\item $|g(z) - B(z)| < \epsilon, \qfor z \in \D$.
\item $\operatorname{deg}(g)= \operatorname{deg}(B).$
\end{enumerate}
\end{theorem}

\emph{Acknowledgments:} 
We would like to thank Dimitrios Betsakos and Misha Lyubich for their comments and observations which initiated this paper, Phil Rippon and Ian Short for useful discussions, and the referee for helpful comments.
%
%
%
%
%
\section{Proof of Proposition~\ref{prop:basics}}
\label{S.basics}
In this section, our goal is to prove Proposition~\ref{prop:basics}. To do this, we first need a little background on inner functions.

It is well-known that it is possible to factorise inner functions in a canonical way. First we define a \emph{singular inner function} as a function of the form
\begin{equation}
\label{eq:singinnerdef}
S(z) \defeq \exp \left(- \int \frac{e^{i\theta} + z}{e^{i\theta} - z} \ d\mu(\theta)\right),
\end{equation}
for some positive and singular measure $\mu$. We then have the following,  which is due to Frostman \cite{Frostman}; see also \cite[p.72]{Garnett} together with \cite[Theorem 6.4]{Garnett}.
\begin{theorem}
\label{theo:inner}
If $g : \D \to \D$ is an inner function, then there is a Blaschke product $B$ and a singular inner function $S$ such that $g = B \cdot S$. Moreover, for all $\zeta \in \D$, except possibly for a set of capacity zero, the function
\begin{equation}
\label{eq:gzetadef}
g_{\zeta}(z) \defeq \frac{g(z) - \zeta}{1 - \overline{\zeta}g(z)},
\end{equation}
is a Blaschke product.
\end{theorem}

We will also use the following, which is a version of \cite[Theorem 4']{Heins}.
\begin{theorem}
\label{theo:heins}
Suppose that $f$ is a {\tef}, that $V$ is a domain, and that $U$ is a component of $f^{-1}(V)$. Then exactly one of the following holds.
\begin{enumerate}[(i)]
\item there exists $n \in \N$ such that $f$ assumes in $U$ every value of $V$ exactly $n$ times (counting multiplicities).
\item $f$ assumes in $U$ every value of $V$ infinitely often with at most one exception.
\end{enumerate}
\end{theorem}

We are now able to prove Proposition~\ref{prop:basics}.
\begin{proof}[Proof of Proposition~\ref{prop:basics}]
That $g$ is an inner function was shown in \cite{EFJS}, in a much less general context, and we repeat the argument for completeness. Suppose that $g$ was not inner. By Fatou's Theorem, there would exist a set $E \subset \partial \mathbb{D}$, of positive measure with respect to $\D$, on which $g$ had non-tangential limits of modulus strictly less than one, and on which $\phi$ had well-defined limits. It would follow that $$\phi(E\setminus \phi^{-1}(\{\infty\})) \subset \partial U$$ was a set of positive harmonic measure with respect to $U$ that was mapped by $f$ into $V$. This is a contradiction, since $f(\partial U) \subset \partial V$.

To see that $g$ can be taken to be a Blaschke product, choose $\zeta \in \D$ such that, by the second part of Theorem~\ref{theo:inner}, the function $g_{\zeta}(z) = \omega(g(z))$ is a Blaschke product, where 
\[
\omega(z) \defeq \frac{z - \zeta}{1-\overline{\zeta}z}.
\]
Set $\tilde{\phi} \defeq \phi \circ \omega^{-1}$ and $\tilde{\psi} \defeq \psi \circ \omega^{-1}$, and so 
\[
\tilde{g} \defeq \tilde{\psi}^{-1} \circ f \circ \tilde{\phi} = \omega \circ \psi^{-1} \circ f \circ \phi \circ \omega^{-1} = g_{\zeta} \circ \omega^{-1},
\]
is a Blaschke product which is associated to $f$.

If $g$ is a finite Blaschke product, of degree $d$ say, then it is easy to see that each point of $\D$ has exactly $d$ preimages up to multiplicity. The statement for $f$ is then immediate, and this gives the case~\ref{theo:a}. 

Otherwise, $g$ is an infinite Blaschke product. In particular $g^{-1}(0)$ is infinite. It follows by Theorem~\ref{theo:heins} that $g^{-1}(\zeta)$ is infinite for all $\zeta \in \D$ except at most one point.
\end{proof}

\section{Singularities of inner functions}
\label{S.singularities}
 The main goal of this section is to prove Theorem \ref{theo:tracts}. 
   Recall that by an \emph{access to infinity} within a domain $U$ we mean 
   a homotopy class of curves tending to infinity within $U$. Any 
   collection of pairwise disjoint curves to infinity comes equipped with
   a natural \emph{cyclic order}, which records how these curves are 
   ordered around $\infty$ according to positive orientation. 
   If $U$ is simply connected, this in turn gives rise to a natural cyclic order 
   on accesses to infinity in $U$. Carath\'eodory's theory of \emph{prime ends},
   see e.g.\ \cite[Chapter~2]{Pommerenke}, provides a natural
   correspondence between accesses to infinity in $U$ and the set of points on
   the unit circle where
   a Riemann map has radial limit $\infty$. Indeed, 
   it follows from the definition that accesses to infinity are in one-to-one correspondence
   with the prime ends represented by a sequence of cross-cuts tending to infinity,
   and therefore the following result follows from 
   \cite[Corollary~2.17]{Pommerenke}. 
   Compare \cite[Correspondence Theorem]{Accesses} for details. 
   
\begin{proposition}\label{prop:correspondence}
 Let $U\subset\C$ be a simply-connected domain, and let $\phi\colon \D\to U$ 
   be a conformal isomorphism. Set 
      \begin{equation}\label{eqn:infinityset}
        \Theta \defeq \{\zeta\in S^1\colon \lim_{t\nearrow 1}\phi(t\zeta) = \infty \} 
      \end{equation}
    For $\zeta\in \Theta$, let $\alpha(\zeta)$ denote the access to infinity in $U$ 
    represented by $\phi([0,1)\cdot \zeta)$.
        Then $\alpha$ is a cyclic-order-preserving 
     bijection between $\Theta$ and the set of accesses to infinity in $U$. 
     
     Moreover,
     if $\gamma\colon [0,\infty)\to U$ is any curve to infinity in $U$ representing
     an access $[\gamma]$, then $f^{-1}(\gamma(t))\to \alpha^{-1}([\gamma])$ as
     $t\to\infty$. 
\end{proposition}

\begin{figure}
\begin{center}
\def\svgwidth{\textwidth}
\input{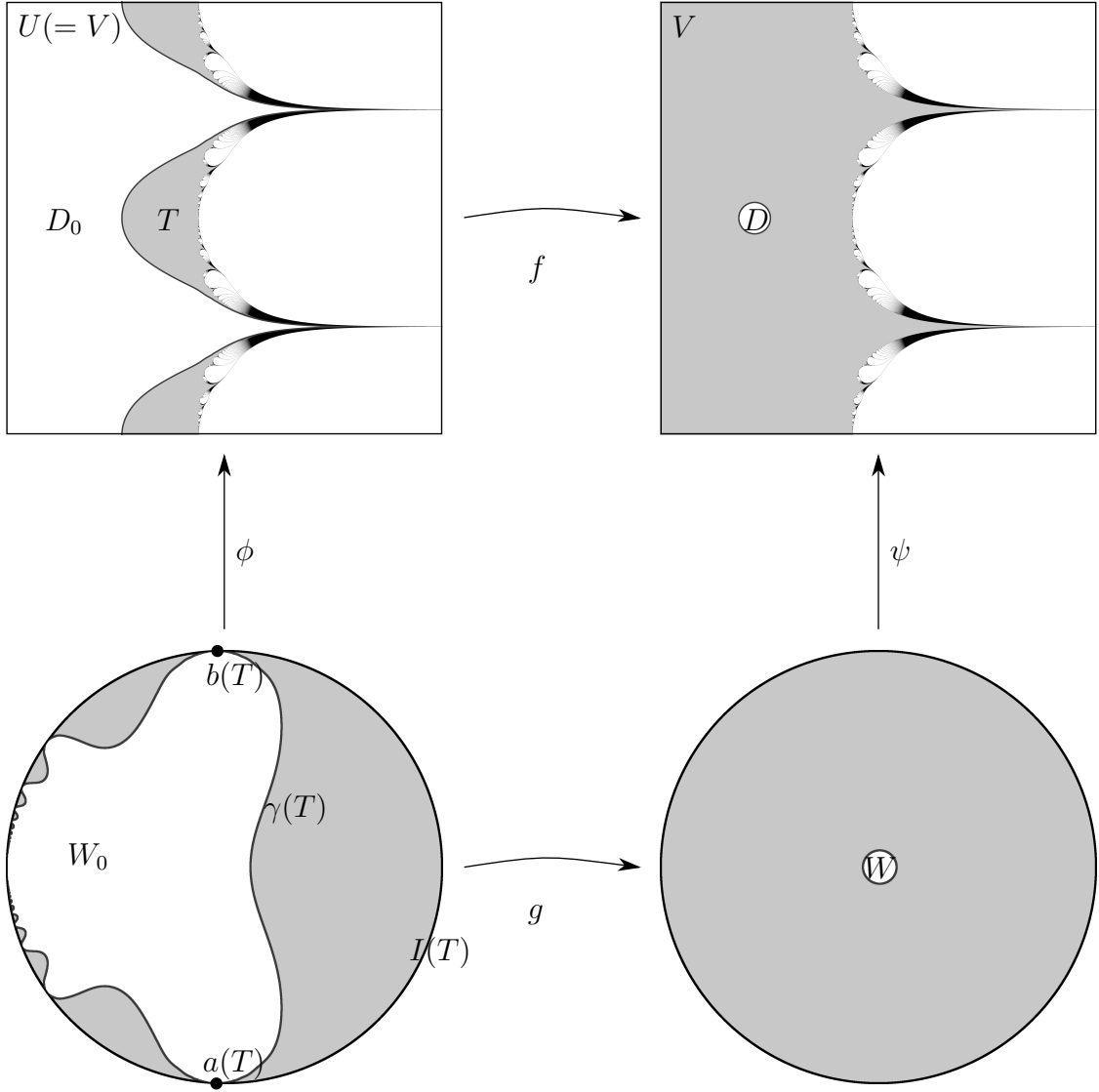}
\end{center}
\caption{\label{fig:singularities}Illustration of the proof of Theorem~\ref{theo:tracts}. Here $f$ 
is the fourth iterate of the exponential map from Figure~\ref{fig:inner}, and $U=V$ is an invariant Fatou component of $f$, whose
 boundary is shown in black.
 Observe that the set $\Theta' = \partial W_0\cap S^1$ of singularities of $g$ is infinite. Also note that $U$ has many more
 accesses to infinity than $D_0$, so that $\Theta\supsetneq \Theta'$; in fact, $\Theta$ is dense in $S^1$.}
\end{figure}

 Observe that Proposition~\ref{prop:correspondence} could also be used as a
   definition of accesses to infinity in $U$ and their cyclic order. 
   The proof of Theorem~\ref{theo:tracts}
   uses some ideas that were also used in the proof of \cite[Theorem 1.1]{EFJS};
   for the reader's convenience, we shall give a largely self-contained account,
   relying only on classical results on the boundary behaviour of
   univalent functions. 
\begin{proof}[Proof of Theorem~\ref{theo:tracts}]
Recall that $f$ is a {\tef},  $V \subsetneq \C$ is a simply-connected domain, and $U$ is a component of $f^{-1}(V)$ such that $f\colon U\to V$ has infinite valence. 
 Finally $D$ is a bounded Jordan domain containing $S(f) \cap V$, such that $\overline{D} \subset V$. 
Set $D_0 \defeq f^{-1}(D)\cap U$. Then $f\colon U\setminus \overline{D_0}\to V\setminus \overline{D}$ is a covering map. Since $V\setminus \overline{D}$ is an annulus, and the map has infinite
 degree, it follows that $f$ is a universal covering when restricted to
  any connected component $T$ 
  of $U\setminus \overline{D_0}$, and that consequently
  $D_0$ is connected, simply connected 
  and unbounded. Compare \cite[Proposition~2.9]{BFR} for details.

Let $\mathcal{T}$ denote the set of components of $U\setminus \overline{D_0}$. 
 With a slight abuse of terminology we call these the \emph{tracts in $U$}.
 Let $T\in\mathcal{T}$. By the above, 
  exactly one of the boundary components of $T$, $\Gamma(T)$ say, is a 
  preimage of $\partial D$, and so is an arc, 
  tending to infinity in both directions,
  which is mapped by $f$ as an infinite-degree covering. 
  
  Now consider Riemann maps $\phi\colon \D \to U$ and $\psi\colon\D \to V$ and an inner function $g \defeq \psi^{-1} \circ f \circ \phi$ associated to $f|_U$.
     Let $\Theta\subset S^1$ be defined as in~\eqref{eqn:infinityset}, and let $\Theta'\subset \Theta$ 
     denote the subset corresponding to accesses to infinity in 
     $D_0$. Note that, by the F.~and M.~Riesz theorem 
     \cite[Theorem~1.7]{Pommerenke}, the set $\Theta$ has
     zero Lebesgue measure and is therefore totally disconnected. 
     Let $X\subset S^1$ denote the set of singularities of $g$; note that $X$ is a compact
      subset of $S^1$. We wish to show
     that $X=\Theta'$, which will be achieved by studying the structure of 
      \[ W_0\defeq \phi^{-1}(D_0) = g^{-1}(W), \] 
      where $W=\psi^{-1}(D)$. (Compare Figure~\ref{fig:singularities}.)
      The set $\partial W_0\cap \D = g^{-1}(\partial W)$ consists of the countably many
      curves $\gamma(T) \defeq \phi^{-1}(\partial T)$ for $T\in\mathcal{T}$. 
      Each $\gamma(T)$ is an arc tending to $S^1$ in both directions. 
      By Proposition~\ref{prop:correspondence}, 
      $\gamma(T)$ in fact has two end-points $a(T),b(T)\in \Theta'$ on the
      unit circle. We may choose the labelling
      such that $\gamma(T)$ separates 
      the arc $I(T)\defeq (a(T),b(T))$ of $S^1$, understood in positive orientation,
      from $W_0$. This implies that
        \begin{equation}\label{eqn:W0boundary}
          \partial W_0 = \bigcup_{T\in\mathcal{T}} \gamma(T) \cup 
               \left(S^1\setminus \bigcup_{T\in\mathcal{T}} I(T)\right).
        \end{equation}
      
  \begin{claim}[Claim~1]
    $a(T),b(T)\in X$ for all $T\in\mathcal{T}$. 
  \end{claim}
  \begin{subproof}
    The restriction $g\colon \gamma(T)\to \partial W$ 
    is a universal covering. In particular, every point of $\partial W$ has infinitely
     many preimages near $a(T)$ and $b(T)$, and these points must be 
     singularities of $g$. 
   \end{subproof}
   
   \begin{claim}[Claim~2]
     The map $g$ extends continuously to each $I(T)$
       as an analytic universal covering $g\colon I(T)\to S^1$. In particular,
       $I(T)\cap X = \emptyset$. 
     \end{claim}
     \begin{subproof}
        The map $g\colon \phi^{-1}(T)\to \D\setminus W$ is a universal covering map.
          Since $\D\setminus W$ is an annulus, the restriction is 
          equivalent, up to analytic changes of coordinate in domain and range, 
          to the restriction of
          the complex exponential map to a horizontal strip $S$, with 
          $a(T)$ and $b(T)$ 
          corresponding to $-\infty$ and $+\infty$ on the boundary of $S$. 
          The conformal isomorphism between $\phi^{-1}(T)$ and $S$ extends 
          continuously to $I(T)$, and thus $g$ extends continuously to $I(T)$
          as a universal covering of $S^1$. By the Schwarz 
          reflection principle, the extension is analytic.
     \end{subproof}
     
   \begin{claim}[Claim~3]
     $W_0$ is a Jordan domain, and $\Theta' = \partial W_0\cap S^1=S^1\setminus\bigcup_T I(T)$. 
   \end{claim}
   \begin{subproof}       
     Recall that the second equality holds by~\eqref{eqn:W0boundary}.
     
     Define $\rho\colon S^1\to \partial W_0$ as follows. On each
        $I(T)$, the map is a homeomorphism 
        $\rho\colon I(T) \to \gamma(T)$, fixing the endpoints $a(T)$ and $b(T)$.
       Outside these intervals, i.e.\ 
        on $\partial W_0\cap S^1$, $\rho$ agrees with the identity. 
        (Note that, by Claim~2, each $I(T)$ is
        a non-degenerate interval, so such $\rho$ does indeed exist.) 
        
        Clearly $\rho$ is injective; it is surjective 
        by~\eqref{eqn:W0boundary}. 
         By definition, the restriction $\rho|_{\overline{I(T)}}$ is continuous for any $T\in\mathcal{T}$, as is $\rho|_{\partial W_0\cap S^1}$. So let 
           $(T_n)_{n=1}^{\infty}$ be a sequence of pairwise different elements of $\mathcal{T}$, and suppose that 
           $\zeta_n\in I(T_n)$ is a sequence converging to $\zeta \in S^1$. To establish continuity of $\rho$, we must prove that 
           $\rho(\zeta_n)\to \zeta$. 
           
          Since the $I(T_n)$ are pairwise disjoint, their lengths tend to zero and $a(T_n)\to \zeta$. 
            Let $\varepsilon>0$. By the Carath\'eodory prime end correspondence, 
            there is a cross-cut $C$ of $\D$, contained in the Euclidean disc 
           $D(\zeta,\varepsilon)$ of radius $\varepsilon$
            around $\zeta$, such that $C$ separates $\zeta$ from $0$ and $\phi(C)$ 
        is a cross-cut of $U$ with finite endpoints
        \cite[Theorem~2.15]{Pommerenke}. 
        Since $\partial D_0\cap\C$ is locally an arc, and $\phi(C)$ is bounded, 
        $\partial T$ intersects $\phi(C)$
        only for finitely many $T$. (Compare \cite[Lemma~2.1]{dreadlocks}.) So, for sufficiently large $n$, $\gamma(T_n)$ is disjoint from $C$, and 
        $C$ separates $a(T_n)$ from $\partial D(\zeta,\varepsilon)$ in $\D$. So $\rho(\zeta_n)\in \gamma(T_n) \subset D(\zeta,\varepsilon)$. We have proved that 
        $\rho(\zeta_n)\to \zeta$, and hence that $\rho$ is continuous. Thus
        $\rho(S^1)=\partial W_0$ is a Jordan curve.
        
       Clearly no $I(T)$ intersects $\Theta'$, and hence $\Theta'\subset \partial W_0\cap S^1$. On the
         other hand, every $\zeta\in \partial W_0\cap S^1$ is accessible from
         $W_0$ by the first part of the claim. If $\gamma \subset W_0$ is a curve tending
         to $\zeta$, then $\phi(\zeta)$ must tend to 
         $\hat{\partial} U \cap \hat{\partial} D_0 = \{\infty\}$ 
         (where $\hat{\partial}$ denotes
         the boundary in the Riemann sphere $\hat{\C}$). Hence $\zeta\in \Theta'$
         by Proposition~\ref{prop:correspondence}, as
         required. 
       \end{subproof}

     Since $\Theta'\subset \Theta$ is totally disconnected, it follows from Claim~3
      that the set $\bigcup_T \{a(T),b(T)\}$ of endpoints of the intervals
      $I(T)$ is dense in $\Theta'$. In particular, by Claim~1, 
      $\Theta'\subset X$, and $X\subset \Theta'$ by Claims~2 and~3. 
      We have established that $X=\Theta'$. 
     
  By Proposition~\ref{prop:correspondence}, the set 
    $\Theta'$ is in one-to-one cyclic-order-preserving correspondence with
    the set of accesses to infinity in $D_0$. Thus we have proved the first
    claim of the theorem. Moreover, clearly $\Theta' = S^1\setminus \bigcup_T I(T)$ is
    finite if and only if $\mathcal{T}$ is finite. If this is the case,
    the map $T\mapsto a(T)$ defines a bijection between $\mathcal{T}$ and
    $\Theta'$. This completes the proof.
    \end{proof}
\begin{proof}[Proof of Corollary~\ref{corr:tracts}]
Suppose that $f \in \B$, that ${S(f)} \subset F(f)$, and that $U$ is an unbounded forward-invariant Fatou component of $f$. Choose a point $w \in U$, and let $D \subset U$ be a hyperbolic disc, centred at $w$, of sufficiently large hyperbolic radius that $S(f) \cap U \subset D$; this is possible since $S(f)$ is compact and does not meet $\partial U$.

It follows from Theorem~\ref{theo:tracts} that the number of singularities of an associated inner function is equal to the number of components of $U \setminus f^{-1}(\overline{D})$. Since $\overline{D} \subset U$, each component of $U \setminus f^{-1}(\overline{D})$ is contained in exactly one component of $\C \setminus (U \cap f^{-1}(\overline{D}))$, so it suffices to count the components of this latter set.
 
Now let $D' \supset D$ be a bounded Jordan domain containing $S(f)$, so that the tracts of $f$ are the components of $\C \setminus f^{-1}(\overline{D'})$. Since $D' \supset D$, no tract can meet more than one component of $\C \setminus (U \cap f^{-1}(\overline{D}))$. However, each component of $\C \setminus (U \cap f^{-1}(\overline{D}))$ meets at least one tract. This completes the proof.
\end{proof}
%
%
%
%

\section{Bounded-degree inner functions}
  In this section, we prove Theorems~\ref{theo:finiteblaschke} and~\ref{theo:wandering}. The results are proved by a standard type of 
   \emph{quasiconformal surgery} \cite[Section~4.2]{brannerfagellasurgery}, 
   which is analogous to the well-known proof of the straightening theorem
   for polynomial-like mappings \cite{douadyhubbardpolynomiallike}. 
   Throughout this section, and only in this section, we shall use
   without comment the standard notions and techniques of 
   quasiconformal surgery, as explained for example in~\cite{brannerfagellasurgery}.

   We begin with Theorem~\ref{theo:wandering}, where the
   surgery takes a particularly simple form. We shall require the following result,
   which establishes the existence of a suitable function on which to perform the surgery.
   
\begin{figure}
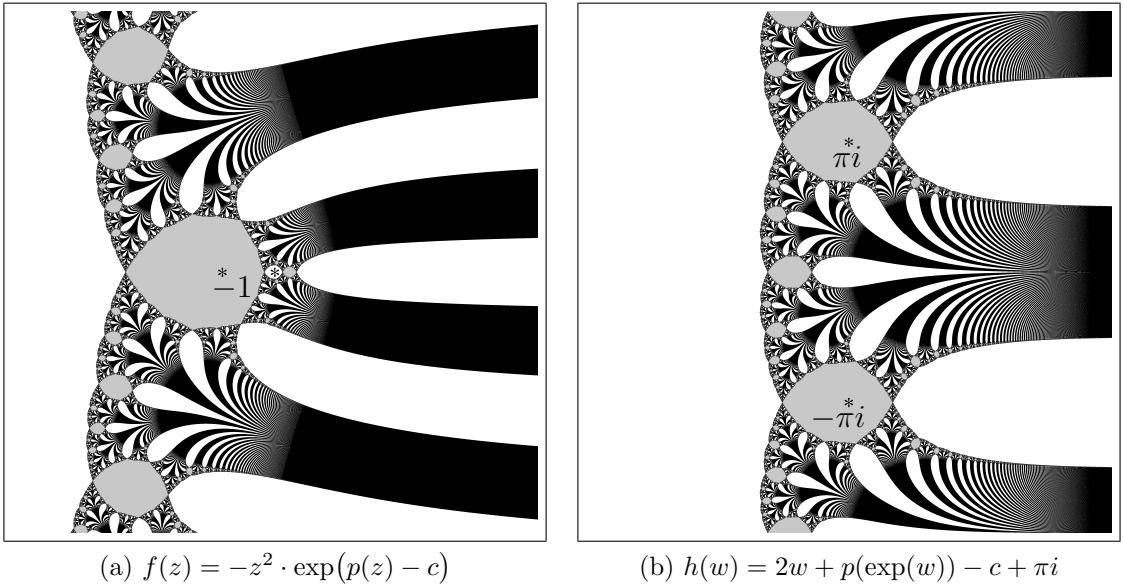

   \subfloat[$f(z)=-z^2 \cdot \exp\bigl(p(z)-c\bigr)$]{\fbox{\def\svgwidth{.47\textwidth}%
\begingroup%
  \makeatletter%
  \providecommand\color[2][]{%
    \errmessage{(Inkscape) Color is used for the text in Inkscape, but the package 'color.sty' is not loaded}%
    \renewcommand\color[2][]{}%
  }%
  \providecommand\transparent[1]{%
    \errmessage{(Inkscape) Transparency is used (non-zero) for the text in Inkscape, but the package 'transparent.sty' is not loaded}%
    \renewcommand\transparent[1]{}%
  }%
  \providecommand\rotatebox[2]{#2}%
  \newcommand*\fsize{\dimexpr\f@size pt\relax}%
  \newcommand*\lineheight[1]{\fontsize{\fsize}{#1\fsize}\selectfont}%
  \ifx\svgwidth\undefined%
    \setlength{\unitlength}{14173.22834646bp}%
    \ifx\svgscale\undefined%
      \relax%
    \else%
      \setlength{\unitlength}{\unitlength * \real{\svgscale}}%
    \fi%
  \else%
    \setlength{\unitlength}{\svgwidth}%
  \fi%
  \global\let\svgwidth\undefined%
  \global\let\svgscale\undefined%
  \makeatother%
  \begin{picture}(1,1)%
    \lineheight{1}%
    \setlength\tabcolsep{0pt}%
    \put(0,0){\includegraphics[width=\unitlength,page=1]{f-deg2-hyperbolic.pdf}}%
    \put(0.3799241,0.45254216){\color[rgb]{0,0,0}\makebox(0,0)[lt]{\lineheight{1.25}\smash{\begin{tabular}[t]{l}$-1$\end{tabular}}}}%
  \end{picture}%
\endgroup%
}}\hfill
   \subfloat[$h(w) = 2w + p(\exp(w)) - c + \pi i$]{\fbox{\def\svgwidth{.47\textwidth}%
\begingroup%
  \makeatletter%
  \providecommand\color[2][]{%
    \errmessage{(Inkscape) Color is used for the text in Inkscape, but the package 'color.sty' is not loaded}%
    \renewcommand\color[2][]{}%
  }%
  \providecommand\transparent[1]{%
    \errmessage{(Inkscape) Transparency is used (non-zero) for the text in Inkscape, but the package 'transparent.sty' is not loaded}%
    \renewcommand\transparent[1]{}%
  }%
  \providecommand\rotatebox[2]{#2}%
  \newcommand*\fsize{\dimexpr\f@size pt\relax}%
  \newcommand*\lineheight[1]{\fontsize{\fsize}{#1\fsize}\selectfont}%
  \ifx\svgwidth\undefined%
    \setlength{\unitlength}{14173.22834646bp}%
    \ifx\svgscale\undefined%
      \relax%
    \else%
      \setlength{\unitlength}{\unitlength * \real{\svgscale}}%
    \fi%
  \else%
    \setlength{\unitlength}{\svgwidth}%
  \fi%
  \global\let\svgwidth\undefined%
  \global\let\svgscale\undefined%
  \makeatother%
  \begin{picture}(1,1)%
    \lineheight{1}%
    \setlength\tabcolsep{0pt}%
    \put(0,0){\includegraphics[width=\unitlength,page=1]{f-deg2-wandering.pdf}}%
    \put(0.53126992,0.20283876){\color[rgb]{0,0,0}\makebox(0,0)[rt]{\lineheight{1.25}\smash{\begin{tabular}[t]{r}$-\pi i$\end{tabular}}}}%
    \put(0.52484824,0.70283877){\color[rgb]{0,0,0}\makebox(0,0)[rt]{\lineheight{1.25}\smash{\begin{tabular}[t]{r}$\pi i$\end{tabular}}}}%
  \end{picture}%
\endgroup%
}}
   \caption{\label{fig:wandering}The function $f$ defined in~\eqref{eqn:Cstarmap} and its lift $h$,
     illustrating Proposition~\ref{prop:wandering} in the case $d=2$. The Julia sets are shown in black and the basin of $-1$,
     respectively its preimage under the exponential map, is shown in grey. Critical points are marked by asterisks.}
\end{figure}

 \begin{proposition}\label{prop:wandering}
  Let $d\geq 2$. Then there exists an entire function $h$ having a simply-connected wandering domain $W$ such that, for all $n\geq 0$,
   $h^n(W)$ is a Jordan domain and 
  $h\colon h^n(W)\to h^{n+1}(W)$ is a proper map of degree $d$. 
 \end{proposition}
\begin{proof}
 We use the well-known
   method of obtaining wandering domains by lifting invariant components
   of a self-map of $\C^* = \C\setminus\{0\}$. Compare \cite[p.~106]{hermanexemples}
   and \cite[p.~414]{sullivanqcdynamicI} for early examples,
    and \cite[p.~3]{classifyingwandering} for a general description of this method. 
    
 Define 
    \begin{align}
         p(z) &\defeq 2\cdot \sum_{j=1}^{d-1}  \binom{d-1}{j} \cdot \frac{z^{j}}{j} \quad\text{and}\notag \\
            f(z) &\defeq -z^2 \cdot \exp\bigl(p(z)-c\bigr), \label{eqn:Cstarmap} \end{align}
    where $c\defeq p(-1)$. 
  Then $f(0)=0$, $f(-1)=-1$, and 
     \begin{align*} f'(z) &= -\exp\bigl(p(z)-c\bigr)\cdot (2z + z^2p'(z)) \\ &=
            -\exp\bigl(p(z)-c\bigr)\cdot 2z\cdot \left(1 + \sum_{j=1}^{d-1} \binom{d-1}{j} \cdot z^j\right) \\ &= 
            - \exp\bigl(p(z)-c\bigr)\cdot 2z\cdot(z+1)^{d-1}.\end{align*}
      Hence $f$ has
   super-attracting fixed points at $0$ and $-1$, and no other critical points. The super-attracting fixed point $0$ is also the
   only asymptotic value of $f$.  
Let $U$ be the Fatou component of $f$ containing $-1$. Since $U$ is a super-attracting basin containing no singular values
other than the fixed point, it follows that 
     $f\colon U\to U$ has degree $d$.
     By~\cite[Theorem~1.10]{BFR}, $U$ is
     a bounded Jordan domain (in fact, a quasidisc). See Figure~\ref{fig:wandering}. 

  Observe that $f^{-1}(0)=\{0\}$. Let  $h$ be the lift of $f$ under $z=\exp(w)$ defined by
    \[ h(w) \defeq 2w + p(\exp(w)) - c + \pi i. \]
    Observe that $h\bigl((2k-1)\pi i\bigr) = (4k-1)\pi i$ for $k\in\Z$.  
     Let $W_k$ be the connected component of $\exp^{-1}(U)$ 
     containing $(2^k-1)\pi i$. By \cite{waltersemiconjugacy}, we have
     $\exp(J(h)) = J(f)$, so $W_k$ is a Fatou component of $h$, with 
     $h(W_k)\subset W_{k+1}$. 
     Since $U$ is a Jordan domain and is
     mapped to itself as a proper map of degree $d$, it follows that
     $W=W_1$ has the desired properties.
     \end{proof} 

\begin{proof}[Proof of Theorem~\ref{theo:wandering}]
   Let $g$ be a finite Blaschke product of degree $d\geq 2$, and let $h$ and $W$ be as in
     Proposition~\ref{prop:wandering}. Let $\tilde{g}$ be the Blaschke product
     associated to $h\colon W\to f(W)$, say $\tilde{g} = \phi_1\circ h \circ \phi_0^{-1}$ with Riemann maps $\phi_0\colon W\to\D$ and $\phi_1\colon f(W)\to \D$. 

 Restricted to $S^1$, both $g$ and $\tilde{g}$ are analytic covering maps of degree $d$, and therefore there is an analytic map 
     $\theta\colon S^1\to S^1$ such that 
        \begin{equation}\label{eqn:gfunctionalrelation} g\circ \theta = \tilde{g}.
        \end{equation} In particular, $\theta$ is quasisymmetric, and can therefore be extended to a quasiconformal
    homeomorphism $\theta\colon \D\to\D$. Define a quasiregular map $\tilde{f}\colon\C\to\C$ by 
      \[ \tilde{f}(z) \defeq \begin{cases} h(z) & \text{if } z\notin W \\
                                                    \phi_1^{-1}(g(\theta(\phi_0(z)))) &\text{if }z\in W. \end{cases}.   \] 
     Since $\partial W$ and $\partial f(W)$ are Jordan curves, 
      the maps $\phi_0$ and $\phi_1$ extend homeomorphically to the
      boundary. By~\eqref{eqn:gfunctionalrelation}, $\tilde{f}$ is continuous at points
        of $\partial W$. By the Bers gluing lemma, it follows that
        $\tilde{f}$ is indeed quasiregular
         on $\C$. 

Let $\mu$ be the dilatation of
                $\theta\circ\phi_0$, which is a Beltrami differential 
                on $U$. Extend $\mu$ to $\C$ as follows:
                If $V$ is a component of the Fatou set of $h$ such that 
                $h^n(V)=W$ for $n>0$, let 
                $\mu|_V$ be the pull-back of $\mu|_{W}$ under $h^n$. On the 
                complement of the backward orbit of $W$, set $\mu=0$. Then $\mu$ is
                invariant under $\tilde{f}$. 

          Apply the measurable Riemann mapping theorem to obtain
             a quasiconformal map $\psi\colon\C\to\C$ whose dilatation is
             $\mu$; then 
               \[ f\defeq \psi\circ \tilde{f}\circ \psi^{-1} \]
               is an entire function. Moreover, consider $U\defeq \psi(W)$ and $V \defeq \psi(h(W))=f(U)$. Then
               \begin{align*} \Phi_0 &\defeq    \theta\circ \phi_0\circ \psi^{-1}               
                   \colon W \to \D \qquad\text{and} \\ 
                               \Phi_1 &\defeq \phi_1\circ \psi^{-1} \colon V\to \D \end{align*}
                are conformal isomorphisms and 
                   \[ \Phi_1 \circ f = \phi_1 \circ \tilde{f}\circ \psi^{-1} = g \circ \theta \circ \phi_0 \circ \psi^{-1} = g \circ \Phi_0 \]
        on $W$.
              So $g$ is an associated inner function of $f\colon U\to V$. 
\end{proof}

The proof of Theorem~\ref{theo:finiteblaschke} is similar. As with 
   Proposition~\ref{prop:wandering}, we should first establish
   the existence of
   suitable subjects for our surgery. (See Figure~\ref{fig:invariantfatou}.)
   
 \begin{figure}
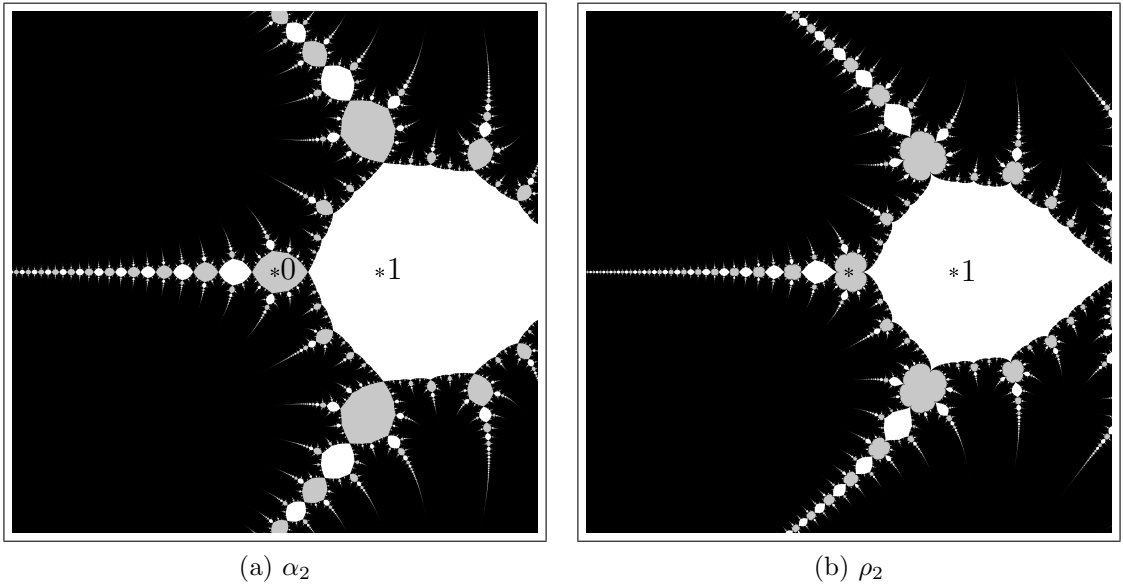

   \subfloat[$\alpha_2$]{\fbox{\def\svgwidth{.47\textwidth}%
\begingroup%
  \makeatletter%
  \providecommand\color[2][]{%
    \errmessage{(Inkscape) Color is used for the text in Inkscape, but the package 'color.sty' is not loaded}%
    \renewcommand\color[2][]{}%
  }%
  \providecommand\transparent[1]{%
    \errmessage{(Inkscape) Transparency is used (non-zero) for the text in Inkscape, but the package 'transparent.sty' is not loaded}%
    \renewcommand\transparent[1]{}%
  }%
  \providecommand\rotatebox[2]{#2}%
  \newcommand*\fsize{\dimexpr\f@size pt\relax}%
  \newcommand*\lineheight[1]{\fontsize{\fsize}{#1\fsize}\selectfont}%
  \ifx\svgwidth\undefined%
    \setlength{\unitlength}{14173.22834646bp}%
    \ifx\svgscale\undefined%
      \relax%
    \else%
      \setlength{\unitlength}{\unitlength * \real{\svgscale}}%
    \fi%
  \else%
    \setlength{\unitlength}{\svgwidth}%
  \fi%
  \global\let\svgwidth\undefined%
  \global\let\svgscale\undefined%
  \makeatother%
  \begin{picture}(1,1.00000039)%
    \lineheight{1}%
    \setlength\tabcolsep{0pt}%
    \put(0,0){\includegraphics[width=\unitlength,page=1]{alpha2-hyperbolic.pdf}}%
    \put(0.71274863,0.48773096){\color[rgb]{0,0,0}\makebox(0,0)[lt]{\lineheight{1.25}\smash{\begin{tabular}[t]{l}$1$\end{tabular}}}}%
    \put(0.51076177,0.48773096){\color[rgb]{0,0,0}\makebox(0,0)[lt]{\lineheight{1.25}\smash{\begin{tabular}[t]{l}$0$\end{tabular}}}}%
  \end{picture}%
\endgroup%
}}\hfill
   \subfloat[$\rho_2$]{\fbox{\def\svgwidth{.47\textwidth}%
\begingroup%
  \makeatletter%
  \providecommand\color[2][]{%
    \errmessage{(Inkscape) Color is used for the text in Inkscape, but the package 'color.sty' is not loaded}%
    \renewcommand\color[2][]{}%
  }%
  \providecommand\transparent[1]{%
    \errmessage{(Inkscape) Transparency is used (non-zero) for the text in Inkscape, but the package 'transparent.sty' is not loaded}%
    \renewcommand\transparent[1]{}%
  }%
  \providecommand\rotatebox[2]{#2}%
  \newcommand*\fsize{\dimexpr\f@size pt\relax}%
  \newcommand*\lineheight[1]{\fontsize{\fsize}{#1\fsize}\selectfont}%
  \ifx\svgwidth\undefined%
    \setlength{\unitlength}{14173.22834646bp}%
    \ifx\svgscale\undefined%
      \relax%
    \else%
      \setlength{\unitlength}{\unitlength * \real{\svgscale}}%
    \fi%
  \else%
    \setlength{\unitlength}{\svgwidth}%
  \fi%
  \global\let\svgwidth\undefined%
  \global\let\svgscale\undefined%
  \makeatother%
  \begin{picture}(1,1)%
    \lineheight{1}%
    \setlength\tabcolsep{0pt}%
    \put(0,0){\includegraphics[width=\unitlength,page=1]{rho2-parabolic.pdf}}%
    \put(0.71274863,0.48373105){\color[rgb]{0,0,0}\makebox(0,0)[lt]{\lineheight{1.25}\smash{\begin{tabular}[t]{l}$1$\end{tabular}}}}%
  \end{picture}%
\endgroup%
}}
   \caption{\label{fig:invariantfatou}The functions $\alpha_2$ and $\rho_2$ from Proposition~\ref{prop:invariantfatouexistence}. Julia sets are shown in black, and the critical points $0$ and $1$ are marked by
    asterisks. The attracting basin of $0$ for $\alpha_2$ and the parabolic basin of $\rho_2$ are shown in grey.}
\end{figure}

 \begin{proposition}\label{prop:invariantfatouexistence}
    For every $d\geq 2$, there is an entire function $\alpha_d$ 
      having an invariant super-attracting Fatou component $W$ which 
      is a bounded Jordan domain, and such that $f\colon W\to W$ is a
      proper map of degree $d$. 
      
     Similarly, there is an entire function $\rho_d$ having an invariant parabolic Fatou 
     component $W$ which 
      is a bounded Jordan domain, and such that $f\colon W\to W$ is a
      proper map of degree $d$. 
 \end{proposition}
 \begin{proof}
  A function $f$ with the properties required of $\alpha_d$ was already
    described in~\eqref{eqn:Cstarmap}, but since we do not require
    $\alpha_d$ to restrict to a self-map of $\C\setminus\{0\}$ here,
    we can also give simpler formulae, such as 
     \[ \alpha_d\colon \C\to\C; \qquad z\mapsto \left(\frac{1-\cos(\pi\sqrt{z})}{2}\right)^d. \]
      Then $S(\alpha_d)=\{0,1\}$, both $0$ and $1$ are super-attracting fixed points, and 
      $0$ is a degree $d$ critical point of $\alpha_d$. 
      
    Let $W$ be the connected component of $F(\alpha_d)$ containing $0$. 
       Then $W$ is simply connected, and 
       $\alpha_d\colon W\to W$ is a branched covering branched only over $0$.
       Since both singular values of $\alpha_d$ belong to (super-)attracting basins,
       the map $\alpha_d$ is hyperbolic in the sense of \cite{BFR}. 
       Again applying  
       \cite[Theorem~1.10]{BFR}, $U_0$ is a quasidisc.

       For $\lambda>0$, define 
         \begin{equation}\label{eqn:flambda} f_{\lambda}(z) \defeq  \frac{\alpha_d(z)+\lambda}{1+\lambda}.\end{equation}
          Then $0$ and $1$ are still critical points of $f_{\lambda}$, with
             critical values $\lambda/(1+\lambda)$ and $1$. Moreover,
             $f_{\lambda}$ is increasing on $[0,1]$. 
             It is easy to see that there 
             is $\lambda_0>0$ such that the orbit of $0$ converges to an attracting
             fixed point for $\lambda<\lambda_0$, to $1$ for $\lambda>\lambda_0$,
             and to a parabolic fixed point for $\lambda=\lambda_0$. (See Figure~\ref{fig:bifurcation-flambda}.) Set
             $\rho_d\defeq f_{\lambda_0}$, and let $W$ be the Fatou component 
             containing $0$. Then $\rho_d\colon W\to W$ is a degree $d$ proper map.
             Moreover, since all singular values belong to attracting or parabolic basins,
             $\rho_d$ is \emph{strongly geometrically finite} in the sense of 
             \cite{geometricallyfinite}. By \cite[Theorem~1.8]{geometricallyfinite},
              $W$ is again a bounded Jordan domain.    
 \end{proof}
\begin{figure}
  \subfloat[$\lambda = 0$]{\includegraphics[width=.3\textwidth]{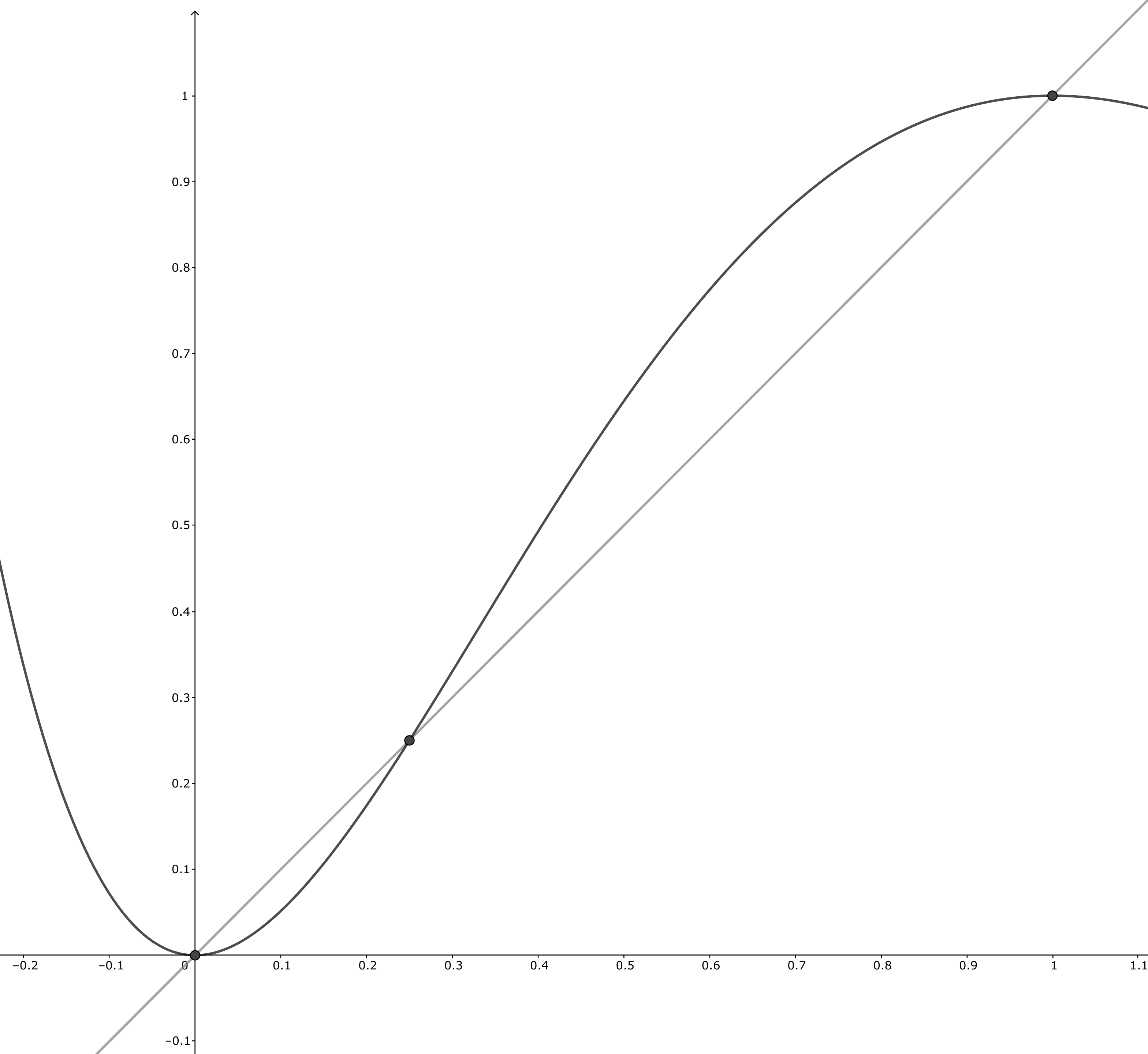}}\hfill%
  \subfloat[$\lambda = \lambda_0$]{\includegraphics[width=.3\textwidth]{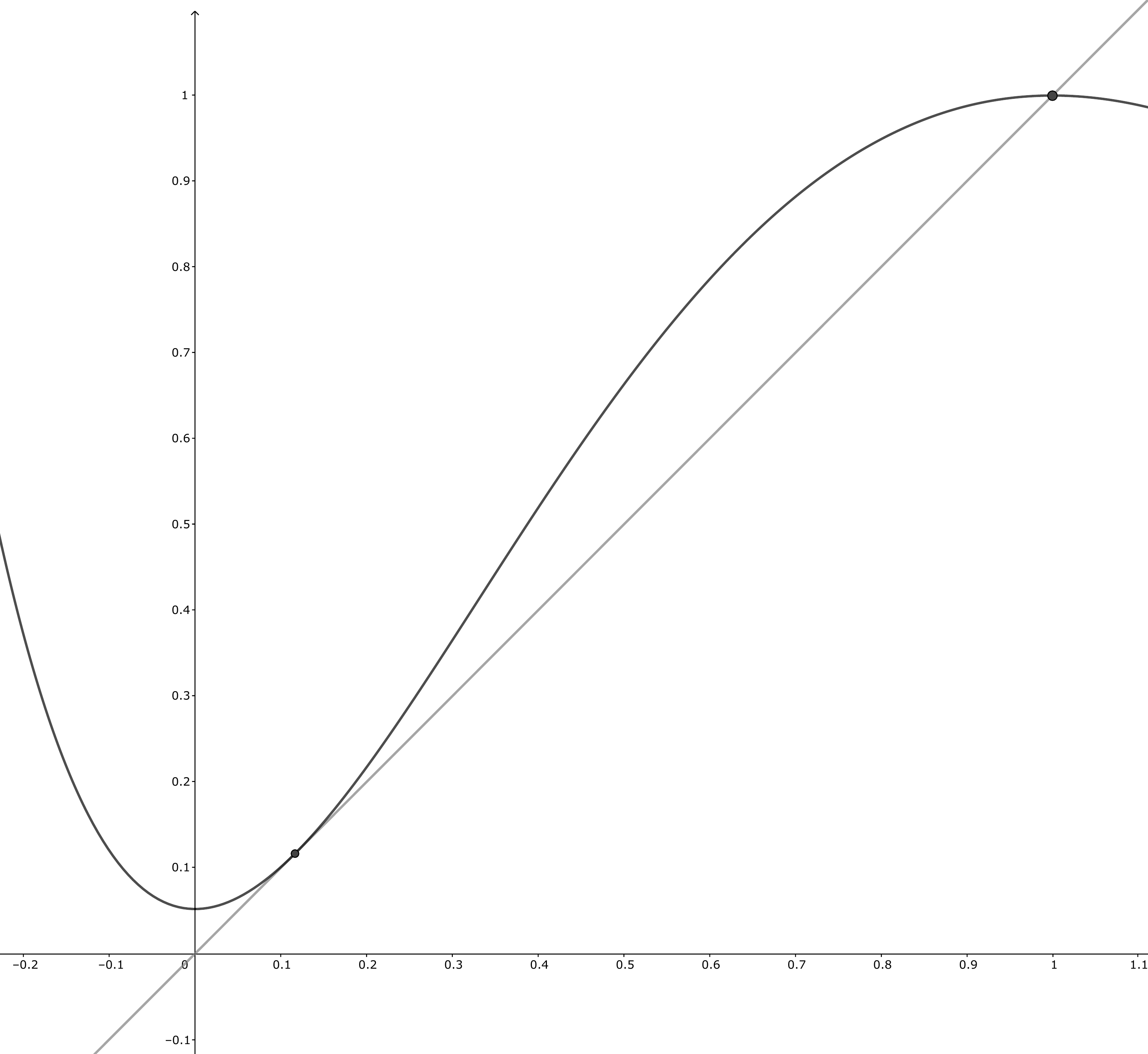}}\hfill%
  \subfloat[$\lambda < \lambda_0$]{\includegraphics[width=.3\textwidth]{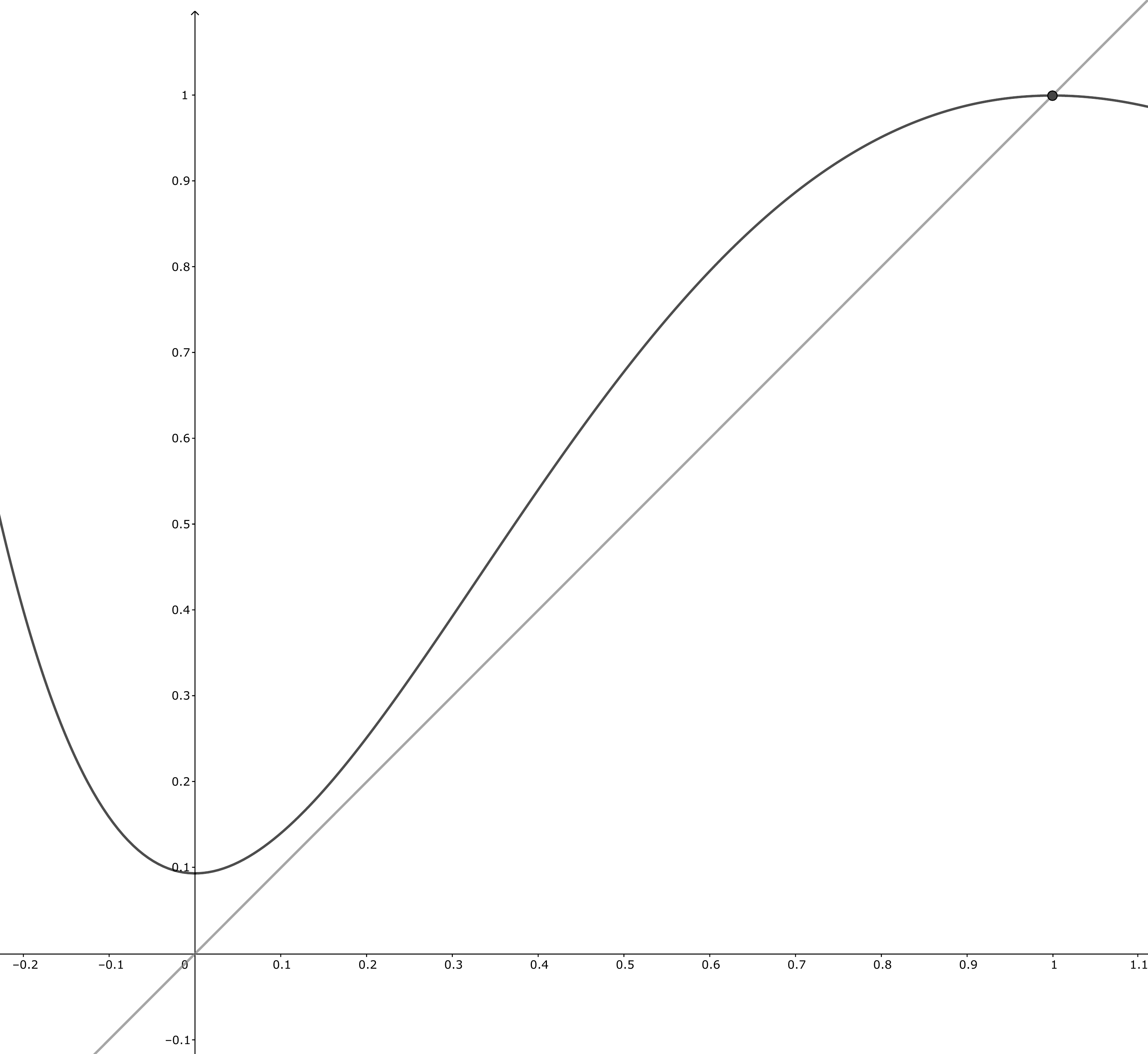}}
  \caption{\label{fig:bifurcation-flambda}The graph of $f_{\lambda}$, as defined in~\eqref{eqn:flambda}, in the case $d=2$. We have $f_0 = \alpha_2$ and $f_{\lambda_0} = \rho_2$. Here
     $\lambda_0\approx 0.0548$.}
\end{figure}

 We now divide the connectivity locus $\mathcal{G}$ of finite Blaschke products
   from Theorem~\ref{theo:finiteblaschke}
   into subclasses as follows. For $d\geq 2$, denote by
   $\mathcal{A}_{d}$ the Blaschke products of degree $d$ having an attracting
   fixed point in $\D$, and by $\mathcal{P}_d$ those having a triple fixed point 
   on $S^1$. The elements of $\mathcal{A}_d$ are called 
   \emph{elliptic} Blaschke products, while those of $\mathcal{P}_d$ are said to be
   parabolic with 
  \emph{zero hyperbolic step}; see \cite{fletcherblaschke}. Then
   $\mathcal{G}=\bigcup_k \mathcal{A}_d\cup \mathcal{P}_d$. 
   Our goal is to extend the proof of Theorem~\ref{theo:wandering} to the case
     of invariant attracting or parabolic components. To do so, we need to be able to
     replace the (non-dynamical) quasisymmetric map $\theta$ by a conjugacy between
     the two Blaschke products in question. This is possible by the following result. 
     (See also
     \cite[Theorem~A]{clarkvanstrienrigidity} for a much more general, and extremely
     deep result.)

     \begin{proposition}\label{prop:conjugacy}
 Let $d\geq 2$. Any two elements of $\mathcal{A}_d$ are quasisymmetrically 
   conjugate on $S^1$, and any two elements of $\mathcal{P}_d$ are quasisymmetrically
   conjugate on $S^1$. 
\end{proposition} 
\begin{proof}
 The maps in $\mathcal{A}_d$ are hyperbolic in the sense of rational dynamics, 
    and hence expanding on their 
    Julia sets. The result is well-known in this case; see \cite[Exercise~2.3 in Chapter~II]{demelovanstrien}, 
     and compare \cite{petersenblaschke} for a more general theorem. See also
     \cite[Section~4.2]{brannerfagellasurgery}.
     
   For $\mathcal{P}_d$, the result
     follows from \cite[Proposition~2.3]{lomonacopetersenshen}. 
\end{proof}

\begin{proof}[Proof of Theorem~\ref{theo:finiteblaschke}]
  Let $g\in\mathcal{G}$, say of degree $d\geq 2$. If $g\in \mathcal{A}_d$, set
     $h \defeq \alpha_d$ from Proposition~\ref{prop:invariantfatouexistence}; 
     if $g\in \mathcal{P}_d$, set $h\defeq \rho_d$. Let
     $W$ be the corresponding invariant Fatou component, let 
     $\phi\colon W\to \D$ be a Riemann map, and set $\tilde{g}\defeq \phi\circ h\circ \phi^{-1}$. 
     
      Then by Proposition~\ref{prop:conjugacy}, there is a quasisymmetric
        homeomorphism $\theta\colon S^1\to S^1$ such that
        $g\circ \theta = \theta\circ\tilde{g}$.  Extend $\theta$ to a quasiconformal
        map $\theta\colon \D\to\D$ and define 
      \[ \tilde{f}\colon \C\to\C;\qquad z\mapsto \begin{cases} h(z) & \text{if } z\notin W \\
                                                    \phi^{-1}\bigl(\theta^{-1}\bigl(g(\theta(\phi(z)))\bigr)\bigr) &\text{if }z\in W. \end{cases}.   \] 
                                                    
      The argument now proceeds 
        exactly as in the proof of Theorem~\ref{theo:wandering}. The function
         $\tilde{f}$ is 
        quasisymmetric. Straightening an invariant Beltrami differential 
        that extends the complex dilatation of $\theta\circ\phi$, we obtain
        an entire function for which $g$ is a dynamically associated inner function.
                     \end{proof}
\begin{remarks}\mbox{}
\begin{enumerate}[(a)]
\item  To carry out the surgery, we could have started with any function 
    $h$ having an invariant attracting or parabolic component $W$ of the 
    required degree $d$. (For simplicity, our proof used the fact that $W$ is
    a Jordan domain, but it is easy to see that this is not essential.) In particular,
    let $\tilde{g}$ be the dynamically associated inner function, and let 
    $g$ be a Blaschke product of the same degree and type (i.e., elliptic or parabolic
    with zero hyperbolic step) as $\tilde{g}$. If additionally $g$ and $\tilde{g}$ are 
    \emph{quasiconformally equivalent}, i.e.\ differ only by pre- and post-composition with
    quasiconformal 
    homeomorphisms of the disc, then there is an entire function $f$ quasiconformally
    equivalent to $h$ for which $g$ is a dynamically associated inner function. 
    As mentioned in the introduction, it seems that a similar result holds for functions
    with finitely many singular values~-- or, more generally, Fatou components $W$
    for which $W\cap S(f)$ is compact. We shall leave this question for discussion in
    a future paper.
    \item Observe that we could also have deduced 
      Theorem~\ref{theo:wandering} from 
      (the proof of) Theorem~\ref{theo:finiteblaschke}, applying the surgery
      for attracting basins
      to the function $f$ 
      defined in~\ref{eqn:Cstarmap}. 
      From this, we obtain another self-map of $\C\setminus\{0\}$ realising a desired 
      Blaschke product. Taking a lift of this second function, we obtain a wandering
      domain with the desired property.
    \item We have restricted to finite-valence
      attracting and parabolic Fatou components,
      where we obtained a complete description of the associated Blaschke
      products. However, let us briefly comment
      on the case of a finite-valence 
      \emph{Baker domain}, i.e.\ an invariant component of
      the Fatou set on which 
      the iterates converge locally uniformly to infinity. Such a domain is 
      called \emph{hyperbolic}, \emph{simply parabolic} or \emph{doubly parabolic}, depending on whether the Denjoy-Wolff point of the
        associated Blaschke product is attracting, a double fixed point, or 
        a triple fixed point. Compare \cite{fagella-henriksen}. 
        Doubly parabolic examples of every finite 
        degree $d$ exist~\cite[Section~4]{fagella-henriksen}. The corresponding
        inner function belongs to $\mathcal{P}_d$, and we can apply our
        surgery to see that every element of $\mathcal{P}_d$ is realised as 
        a dynamically associated inner function of an entire function with a 
        Baker domain. 
        
      An analogue of 
        Proposition~\ref{prop:conjugacy} also holds for Blaschke products 
        with $J(f)\neq S^1$ 
        (again, this is a simple special 
        case of~\cite[Theorem~A]{clarkvanstrienrigidity}). Therefore, starting
        with a function having a hyperbolic or simply-parabolic Baker domain
        of finite degree, we can apply the same surgery technique. However,
        as far as we are aware, the only known hyperbolic
        and simply-parabolic Baker domains of finite valence 
        are univalent. Hence we cannot presently
        answer the question which finite Blaschke products
        arise as dynamically associated inner functions of entire functions with
        Baker domains.
  \end{enumerate}
\end{remarks}

\section{Inner functions of exponential maps}\label{sec:expo}
 
We begin with the following well-known observation concerning
  exponential maps $f_{\lambda}(z)=\lambda e^z$. 

\begin{lemma}[{\cite[Lemma~1.1]{DevandG}}]\label{lem:attractingexp}
  $f_{\lambda}$ has a fixed point of multiplier $\tau\in\C\setminus\{0\}$ 
    if and only if $\lambda = \tau\cdot e^{-\tau}$. 
\end{lemma}
\begin{proof}
  If $z$ is a fixed point of multiplier $\tau$, then 
  $z=f_{\lambda}(z) =f_{\lambda}'(z)=\tau$. So 
  $\tau = \lambda e^{\tau}$, as claimed. 
\end{proof}

The following shows that unisingular inner functions with an attracting
  fixed point are determined by their degree and their multiplier.

\begin{lemma}\label{lem:unisingular}
  Let $\tau\in\D\setminus\{0\}$, and let $2\leq d\leq \infty$. Then, up to conjugacy by a M\"obius automorphism of $\D$, there exists a unique inner function
   $g\colon\D\to\D$ of degree $d$ such that $g$ has an attracting fixed point of multiplier $\tau$ in $\D$, and such that $g$ has only one singular value in $\D$. 
\end{lemma}
\begin{proof}
  To prove existence, it is enough to exhibit the existence of a polynomial 
     or entire function
     having an attracting fixed point of multiplier $\tau$, and having only
     one singular value. Indeed, then the dynamically associated inner function
     is of the stated form. 
     
    For $d<\infty$, such a function is given by the polynomial 
      \[ z\mapsto \frac{\tau}{d} \cdot ( (z+1)^d-1); \]
     for $d=\infty$ the function $f_{\tau\cdot e^{-\tau}}$ has the 
     desired properties by  Lemma~\ref{lem:attractingexp}. 

  So it remains to prove uniqueness. Suppose that $g$ and $\tilde{g}$ are both functions with the stated properties. Let $z_0$ and $\tilde{z}_0$ be the corresponding fixed points,
     and $s$ and $\tilde{s}$ the singular values. By the K{\oe}nigs linearisation theorem, there is a simply-connected domain $U_0\supset \{z_0,s\}$ and a 
     conformal isomorphism
      \[ \phi \colon U_0 \to B(0,\lvert\tau\rvert^{-1}) \]
      such that $\phi(g(z)) = \tau \phi(z)$ and $\phi(s)=1$. An analogous function $\tilde{\phi}$ on a domain $\tilde{U}_0$ exists also for $\tilde{g}$. Set
        \[ \psi_0\defeq \tilde{\phi}^{-1}\circ \phi \colon U_0 \to \tilde{U}_0. \]
       Then $\psi_0$ conjugates $g$ to $\tilde{g}$ on $U_0$, with $\psi_0(z_0)=\tilde{z}_0$ and $\psi_0(s) =\tilde{s}$. 

    Now set $U_1\defeq g^{-1}(U_0) \supset U_0$. Then $g\colon U_1\to U_0$ is either a branched covering of degree $d$, branched only over $s$ (if $d<\infty$), or
      a universal covering (otherwise); see \cite[Proposition 2.8]{BFR}. 
      The same is true for $\tilde{g}$. It follows that we can lift $\psi_0$ to a map
      $\psi_1\colon U_1\to \tilde{U}_1$ such that $\psi_0(g(z)) = \tilde{g}(\psi_1(z))$ and 
      $\psi_1(z_0)=\tilde{z}_0$. We have $\psi_1(z)=\psi_0(z)$ near $z_0$, and hence by the identity theorem on all of $U_0$. In particular,
      $\psi_1(s)=\tilde{s}$, and we can continue inductively. 

     In this manner, we obtain a conformal conjugacy $\psi$ between $g$ and $\tilde{g}$ on 
       \[ \bigcup g^{-n}(U_0) = \D. \] 
     In other words, $g$ and $\tilde{g}$ are M\"obius conjugate, as claimed. 
\end{proof}

We now study the family of maps $g_{a,b}$ as in~\eqref{eqn:tangent}. (See Figure~\ref{fig:tan_inner}.)
  
\begin{figure}
\begin{center}
\def\svgwidth{.9\textwidth}
\begingroup%
  \makeatletter%
  \providecommand\color[2][]{%
    \errmessage{(Inkscape) Color is used for the text in Inkscape, but the package 'color.sty' is not loaded}%
    \renewcommand\color[2][]{}%
  }%
  \providecommand\transparent[1]{%
    \errmessage{(Inkscape) Transparency is used (non-zero) for the text in Inkscape, but the package 'transparent.sty' is not loaded}%
    \renewcommand\transparent[1]{}%
  }%
  \providecommand\rotatebox[2]{#2}%
  \newcommand*\fsize{\dimexpr\f@size pt\relax}%
  \newcommand*\lineheight[1]{\fontsize{\fsize}{#1\fsize}\selectfont}%
  \ifx\svgwidth\undefined%
    \setlength{\unitlength}{498.70861321bp}%
    \ifx\svgscale\undefined%
      \relax%
    \else%
      \setlength{\unitlength}{\unitlength * \real{\svgscale}}%
    \fi%
  \else%
    \setlength{\unitlength}{\svgwidth}%
  \fi%
  \global\let\svgwidth\undefined%
  \global\let\svgscale\undefined%
  \makeatother%
  \begin{picture}(1,0.59178445)%
    \lineheight{1}%
    \setlength\tabcolsep{0pt}%
    \put(0,0){\includegraphics[width=\unitlength,page=1]{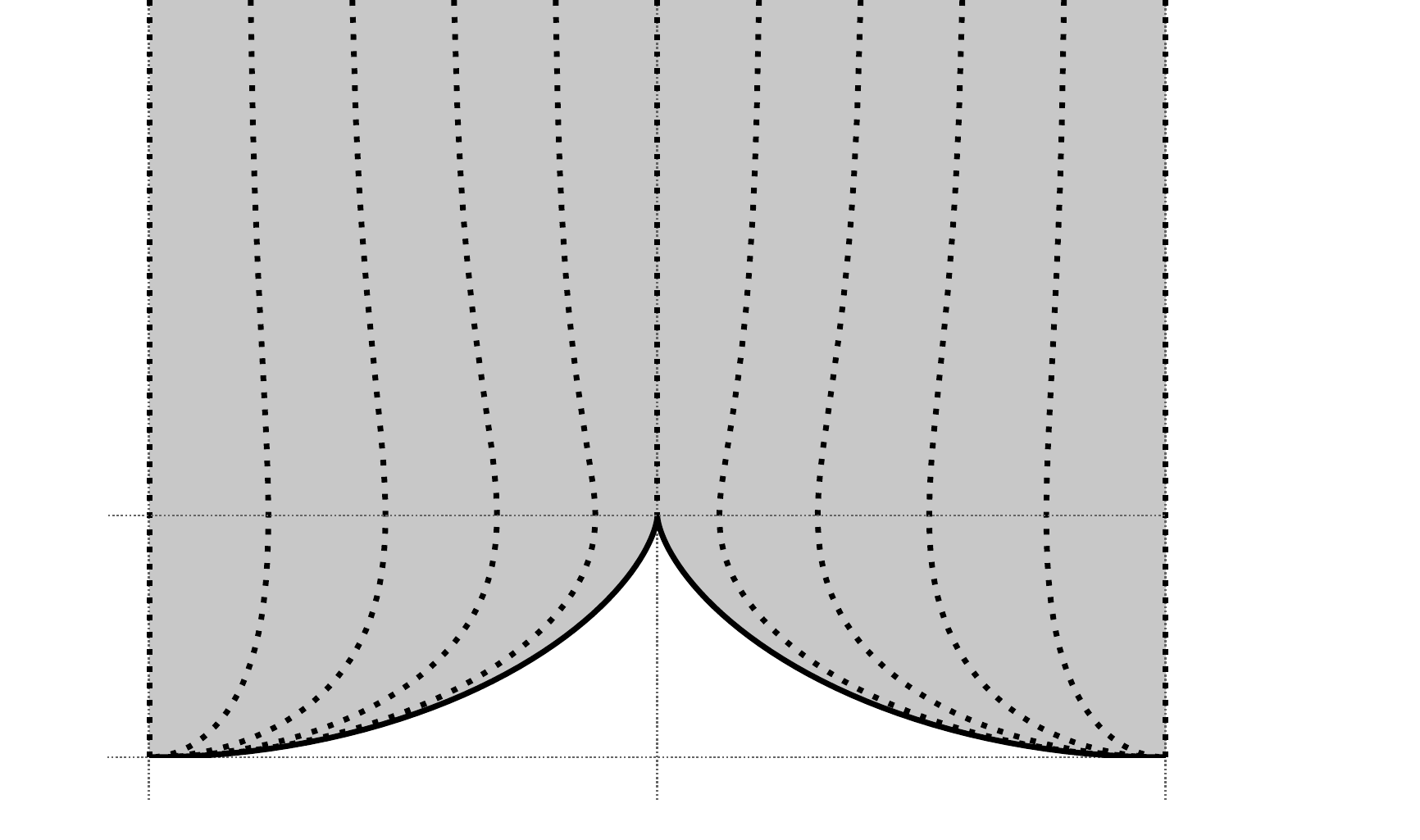}}%
    \put(-0.00182111,0.2343357){\color[rgb]{0,0,0}\makebox(0,0)[lt]{\lineheight{1.25}\smash{\begin{tabular}[t]{l}$a=1$\end{tabular}}}}%
    \put(-0.00182111,0.06138902){\color[rgb]{0,0,0}\makebox(0,0)[lt]{\lineheight{1.25}\smash{\begin{tabular}[t]{l}$a=0$\end{tabular}}}}%
    \put(0.4294186,0.00750834){\color[rgb]{0,0,0}\makebox(0,0)[lt]{\lineheight{1.25}\smash{\begin{tabular}[t]{l}$b=0$\end{tabular}}}}%
    \put(0.06413332,0.00484063){\color[rgb]{0,0,0}\makebox(0,0)[lt]{\lineheight{1.25}\smash{\begin{tabular}[t]{l}$b=-\frac{\pi}{2}$\end{tabular}}}}%
    \put(0.77108063,0.00484063){\color[rgb]{0,0,0}\makebox(0,0)[lt]{\lineheight{1.25}\smash{\begin{tabular}[t]{l}$b=\frac{\pi}{2}$\end{tabular}}}}%
  \end{picture}%
\endgroup%

\end{center}
\caption{\label{fig:tan_inner}The family $g_{a,b}$, for $-\pi/2<b<\pi/2$ and $0<a<\pi$. For parameters in the grey region,
  $g_{a,b}$ has an attracting fixed point in $\HH$. The curve 
   $\lvert b\rvert = \arccos(\sqrt{a})-\sqrt{a}\cdot \sqrt{1-a}$ is shown in black. Note that $g_{a,b}$ has a multiple
   fixed point in $\R$ for each parameter on this curve, but only $\tan = g_{1,0}$ has a triple fixed point. 
    The strong dotted lines are curves of fixed argument for the 
    multiplier of the attracting fixed point of $g_{a,b}$.}
\end{figure}
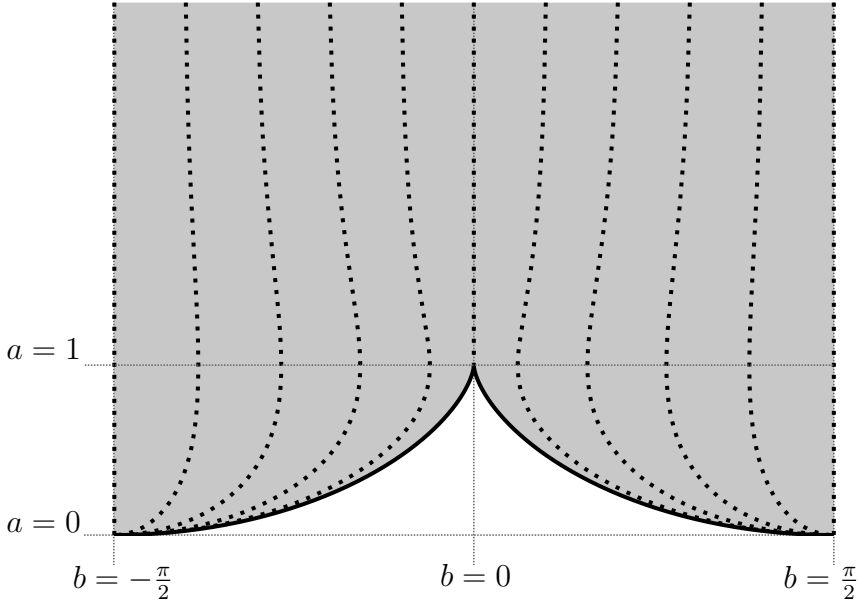  

\begin{proposition}\label{prop:G}
   No two different maps $g_{a,b}$ are conformally conjugate. Moreover, $g_{a,b}$ has an attracting fixed point in $\HH$ if and only if $a>1$ or $a\leq1$ and
     $\lvert b\rvert > \arccos(\sqrt{a})-\sqrt{a}\cdot \sqrt{1-a}$.
\end{proposition}
\begin{proof}
     If $g_{a,b}$ and $g_{\tilde{a},\tilde{b}}$ are conformally conjugate, then the conjugacy $\psi$ must preserve the set 
       $g_{a,b}^{-1}(\infty) = g_{\tilde{a},\tilde{b}}^{-1}(\infty)$, which consists of the odd multiples of $\pi/2$. So $\psi$ 
         is a translation by an integer multiple of $\pi$. Since it must
       also map singular values to singular values, we have $\psi(ai+b) = \tilde{a}i + \tilde{b}$. So $a=\tilde{a}$ and $b-\tilde{b}\in \pi \Z$. As $b,\tilde{b} \in (-\pi/2,\pi/2]$, we see that 
     $ b = \tilde{b}$ as required. 

    Recall that, by the Denjoy-Wolff theorem, for every $a$ and $b$ 
      there is a point $\zeta_0\in\overline{\HH}\cup\{\infty\}$ such that
      $g_{a,b}^n\to \zeta_0$ locally uniformly on $\HH$. We claim that
      $\zeta_0\neq \infty$. Indeed, recall that $g_{a,b}$ is $\pi$-periodic,
      and that $g_{a,b}(z)\to b + ai\in\HH$ as $\im z\to +\infty$. 
      Hence, if $z_n\defeq g_{a,b}^n(z_0)\to\infty$ for some
      $z_0\in\overline{\HH}$, we must have $\im z_n\to 0$. 
      On the other hand, within any horizontal strip of bounded height we have 
          \[ \lvert \tan'(z)\rvert = \frac{1}{\lvert \cos(z)\rvert^2} \geq 
              C\cdot \lvert \tan(z)\rvert^2 \]
      for some constant $C$. So, in particular, 
      $\lvert g_{a,b}'(z_n)\rvert \to \infty$. It follows that
      \[ \im z_{n+1} = \dist(z_{n+1},\R) \geq \dist(z_n,\R) = \im z_n \]
      for sufficiently large $n$. Since $\im z_n\to 0$ and $\R$ is completely invariant,
      this is possible only if $z_0\in\R$. 
      
 Hence $\infty$ cannot be the
      Denjoy-Wolff point of $g_{a,b}$. 
            In particular, 
        $g_{a,b}$ has an attracting fixed point in $\HH$ if and only if it does not 
        have an attracting or parabolic fixed point in $\R$. 

    Now, if $g_{a,b}$ has a fixed point of multiplier $\tau>0$ at $\alpha\in\R$, then  $\alpha$ is not an odd multiple of $\pi/2$, and 
      \[  \tau = g_{a,b}'(\alpha) = \frac{a}{(\cos \alpha)^2}. \]
     So $a=a_{\tau}(\alpha) = \tau\cdot (\cos \alpha)^2$, and in particular $a\leq \tau$.  But also 
     \[ \alpha = g_{a,b}(\alpha) = \tau \cdot (\cos \alpha)^2\cdot \tan \alpha+b= \tau\cdot \cos \alpha\cdot \sin\alpha + b, \]
     so
       \[ b = b_{\tau}(\alpha) = \alpha - \tau\cdot \cos\alpha \cdot \sin\alpha. \]
     Note that $a(\alpha)$ is a strictly decreasing function of $\alpha$ on $[0,\pi/2)$, and it is an easy exercise to see that, for $\tau\leq 1$, 
       $b_{\tau}(\alpha)$ is a strictly increasing function of $\alpha$. 

    To prove the claim, let us restrict to the case $b\in [0,\pi/2]$, which we can do by symmetry. If $g_{a,b}$ has a parabolic  fixed point $\alpha$ 
      in $\R$, then $a \leq 1$ and $\alpha\in[0,\pi/2]$. Therefore 
     \[ b = \alpha - \cos \alpha \cdot \sin\alpha = \arccos\sqrt{a} - \sqrt{a}\cdot \sqrt{1-a} =: \theta(a). \]
     Moreover, for fixed $\alpha$, $a_{\tau}(\alpha)$ is an increasing function of $\tau$, while $b_{\tau}(\alpha)$ is a decreasing function of $\tau$. 
      It follows that $g_{a,b}$ has an attracting fixed point in $\R$ if and only if $a<1$ and $b<\theta(a)$, as claimed.
\end{proof}

\begin{proof}[Proof of Theorem~\ref{theo:unisingular}]
  Let $f$ and $U$ be as in the theorem, and let $g\colon\HH\to\HH$ be an inner function dynamically associated to $f^n$ on $U$. 
    Then $g$ has an attracting fixed point in $\HH$, and a single singular value in $\HH$, which we may assume to be placed at $i$. Then 
     $g\colon \HH\to \HH\setminus \{i\}$ is a universal covering map. So is $\tan$, and the two maps agree up to pre-composition by a M\"obius transformation of the half-plane.
     Applying a suitable M\"obius conjugacy to $g$, we see that it can be chosen of the form $g_{a,b}\in\mathcal{G}$. 
\end{proof}

\begin{proof}[Proof of Theorem~\ref{theo:exp}]
   The characterisation of $\mathcal{G}$ is in Proposition~\ref{prop:G}, and
      claim~\ref{p1} holds by assumption. If $f\in\mathcal{F}$, then by Theorem~\ref{theo:unisingular}, there is an inner function
       $g\in\mathcal{G}$ dynamically associated to $f$, and this function is unique by the first part of Proposition~\ref{prop:G}.  
     Finally, let $g\in\mathcal{G}$ have an attracting fixed point of multiplier $\tau\in\D\setminus\{0\}$ in $\D$. There is a unique $\lambda\in\C\setminus\{0\}$ such that 
     $f_{\lambda}$ has a fixed point 
      of multiplier $\tau$, namely $\lambda = \tau\cdot e^{-\tau}$. As we have just proved, there is a dynamically associated inner function $g_{a,b}\in\mathcal{G}$, and by
      Lemma~\ref{lem:unisingular} and the first part of Proposition~\ref{prop:G}, we have $g=g_{a,b}$, as required. 
\end{proof} 

For the parabolic case, we use the following version of Lemma~\ref{lem:unisingular}.

 \begin{lemma}\label{lem:unisingularparabolic}
    Up to conformal conjugacy, $\tan\colon\HH\to\HH$ is the only inner 
       function of infinite valence 
       that has a unique singular value in $\HH$ and that has
       a fixed point of multiplicity $3$ in $\R$. 
 \end{lemma}
 \begin{proof}
   The proof is similar to Lemma~\ref{lem:unisingular}: Given two
      functions $g$ and $\tilde{g}$ with the stated properties, we can use
      Fatou coordinates to construct petals $U_0,\tilde{U}_0\subset \HH$ of
       $g$ and $\tilde{g}$ and a conformal isomorphism 
       $\psi_0\colon U_0\to\tilde{U}_0$ with $\psi_0\circ g = \tilde{g}\circ \psi_0$, 
       and such that $\psi_0(s) = \tilde{s}$. (Here $s$ and $\tilde{s}$ are
       again the singular values of $g$ and $\tilde{g}$.) 
       
      We can again lift $\psi_0$ to a map $\psi_1$ on $g^{-1}(U_0)$, 
        chosen such that $\psi_1(g(s))=\tilde{g}(\tilde{s})$, and it follows as 
        before that $\psi_0 = \psi_1$ on $U_0$. The proof now proceeds as before,
        and we conclude that $g$ and $\tilde{g}$ are M\"obius conjugate.
 \end{proof} 

\begin{proof}[Proof of Theorem~\ref{theo:unisingularparabolic}]
  Suppose that $f$ and $U$ satisfy the hypotheses of the theorem.
   Let $\phi\colon U\to\HH$ be a conformal isomorphism, and let
   $g \defeq \phi\circ f^n\circ \phi^{-1}$ be the dynamically associated inner
   function, where $n$ is the period of the
   parabolic Fatou component $U$.
   
  Recall that $g$ is of infinite valence and
   has only one singular value $\alpha\in\HH$. 
   Consequently $g\colon \HH\to \HH\setminus\{\alpha\}$ is a universal
   covering map. Applying a suitable M\"obius conjugacy, 
   $g$ can be taken of the form $g=g_{a,b}$ 
   for unique choices of $a>0$ and $b\in (-\pi/2,\pi/2]$.

  Since $f^n$ has no fixed point in $U$, the inner function 
   $g$ has no fixed point in $\HH$, so its Denjoy-Wolff point $\zeta_0$
   must lie on the boundary.  
  As noted in the proof of Proposition~\ref{prop:G}, we have 
   $\zeta_0\neq \infty$. The point $\zeta_0$ is thus either an attracting fixed point,
   a double fixed point (with a single attracting direction along $\R$) 
   or a triple fixed point (with two repelling directions along the real axis). 
   As mentioned previously, the final case holds if and only if 
   $g$ has \emph{zero hyperbolic step}; i.e.\ 
      \[ \dist_{\HH}(g^k(z),g^{k+1}(z))\to 0  \]
   as $k\to\infty$, for all $z\in\HH$. 
   Here $\dist_{\HH}$ denotes hyperbolic distance. 
   
  It is well-known that 
   \[ \dist_U(f^{kn}(z),f^{(k+1)n}(z)) \to 0 \] 
   for $z\in U$. 
   Indeed, the proof of the existence
   of Fatou coordinates, see \cite[Chapter~10]{milnor}, shows that 
   all $z\in U$ eventually enter an \emph{attracting petal} $P\subset U$ on 
   which $f^n$ is conformally conjugate to the map $z\mapsto z+1$ on the
   upper half-plane. So $f^{kn}(z)\in P$ for large enough $k$, and 
     \[ \dist_U(f^{kn}(z),f^{(k+1)n}(z)) \leq \dist_P(f^{kn}(z),f^{(k+1)n}(z) =
         O(1/k). \] 
         
    Since $g$ is conformally conjugate to $f^n|_U$, we see that $g$ does indeed
      have zero hyperbolic step, and hence a triple fixed point at $\zeta_0$. 
      By Lemma~\ref{lem:unisingularparabolic}, we have $g=\tan$, as claimed.
\end{proof}
\begin{remark}
 In the specific case of the parabolic exponential map $f(z)=e^{z-1}$, we could
  proceed somewhat more directly, using the inherent symmetry of the Julia set. 
  Indeed, the parabolic basin $U$ intersects the real axis in the interval
  $(-\infty,1)$. This interval is a hyperbolic geodesic of $U$ by symmetry, 
  and contains the singular value $0$. For the inner function
   $h=h_{a,b}\colon \HH\to\HH$, it follows that the hyperbolic
  geodesic connecting the Denjoy-Wolff point to $\infty$ contains the 
  singular value. From this, we easily conclude that $b=0$, so that
  $h= a\cdot \tan(z)$, with $a\leq 1$. We have 
     \[ \dist_{\HH}(z,a\cdot \tan(z)) \asymp \dist_{\HH}(z,a\cdot z) = \log 1/a \]
  as $z\to 0$ in $\HH$. Hence we must have $a=1$, as claimed. 
\end{remark}

We note that similar results to Theorems~\ref{theo:unisingular} 
  and~\ref{theo:unisingularparabolic} hold when the cycle of $U$ contains only one
   singular value and and $f^n\colon U\to U$ is proper of degree $d$. 
   Indeed, in this case the associated inner function $g$ is a finite-degree
   \emph{unicritical} Blaschke product having connected Julia set, and the 
   connectedness locus of unicritical Blaschke products has been
   described in detail in~\cite{fletcherblaschke,fletcheretalblaschke}. When
   $d=2$, an elliptic Blaschke product fixing zero with multiplier $\lambda$ is given by 
     \[ z\mapsto z\cdot \frac{z+\lambda}{1+\overline{\lambda}z}, \] 
     see \cite[Section~4.2.1]{brannerfagellasurgery}, and a Blaschke product
     with a parabolic fixed point is given by the function~\eqref{eqn:parabolicblaschke}. 
     Hence, if $U$ is a periodic Fatou component of period $n$ whose orbit
     contains just one singular value, which is a critical value of degree $2$,
     then one of the above Blaschke products is dynamically associated to
     $f^n\colon U\to U$. 

\section{A generalisation of exponential maps}
\label{S.exponentials}
  We now generalise our considerations for exponential maps as follows.

\begin{theorem}\label{theo:finite}
  Suppose that $f$ is an entire function and $U$ is an unbounded forward-invariant Fatou component on which $f$ has infinite valence, but such that 
    $f^{-1}(a)\cap U$ contains exactly $p$ points, counting multiplicity, for some $a\in U$ and $p\geq 0$. 
      Assume that an 
     inner function dynamically associated to $f|_U$ has a finite number $q\geq 1$ of singularities on $\partial \D$. 
Then $f$ has a dynamically associated inner function of the form
\begin{equation}
\label{eq:Sdef}
g\colon \D\to\D; \qquad z \mapsto B(z) \exp\left(- \sum_{j=1}^q \left(c_j \frac{e^{i\theta_j}+z}{e^{i\theta_j} - z}\right)\right),
\end{equation}
for some finite Blaschke product $B$ of degree $p$, real numbers $\theta_1, \ldots, \theta_q$, and positive real numbers $c_1, \ldots, c_q$. 
\end{theorem}

Before we prove the theorem, let us note a special case. 

\begin{cor}
\label{cor:exponentials}
Suppose that $P, Q$ are polynomials of degree 
$\deg P \geq 0$ and $\deg Q \geq 1$. Suppose also that the function
\[
f(z) \defeq P(z) e^{Q(z)},
\]
has an unbounded forward-invariant Fatou component, $U$, containing the origin, on which $f$ has infinite valence. 
Then $f$ has a dynamically associated inner function of the form~\eqref{eq:Sdef}, with $q\leq \deg Q$ and $p\leq \deg P$. 
\end{cor}
\begin{remark}\normalfont
If $\lambda$ is sufficiently small, then the conditions of this corollary hold for $\lambda f$; see, for example, \cite[Lemma 7.1]{DaveSurvey}. 
\end{remark}
\begin{proof}[Proof of Corollary~\ref{cor:exponentials} using Theorem~\ref{theo:finite}]
  There are at most $\deg P$ preimages of $0$ under $f$ in $U$ 
   (counting multiplicity), and any associated inner function has at most 
   $\deg Q$ singularities on
    $\partial\D$ by Theorem~\ref{theo:tracts} (note that $S(f)$ is finite, and hence $U\cap S(f)$ is compact). Hence the hypotheses of Theorem~\ref{theo:finite} are satisfied. 
\end{proof}

\begin{proof}[Proof of Theorem~\ref{theo:finite}]
Let $g$ be an inner function dynamically associated to $f$. We shall assume that the Riemann map $\phi \colon \D \to U$ is chosen so that $\phi(0) = a$. By Theorem~\ref{theo:inner}, set $g = B \cdot S$, where $B$ is a Blaschke product and $S$ is a singular inner function of the form \eqref{eq:singinnerdef}.

Since $S$ is never zero, and $a$ has exactly $p$ preimages under $f$, counting multiplicity, it follows that $B$ must be a finite Blaschke product of degree $p$. 
   Note that this implies that $B$ has no singularities in the boundary of the disc.

It is easy to see, for example, by \cite[Theorem 6.2]{Garnett}, that the singularities of $g$ correspond exactly to the support of $\mu$. Since $g$ has only $q$ singularities, there exist real numbers $\theta_1, \ldots, \theta_q$ and positive real numbers $c_1, \ldots, c_q$ such that $\mu$ is equal to $q$ point masses, each of mass $c_j$, at the points $e^{i\theta_j}$. The result follows.
\end{proof}
We now give three applications of this result. Our first example
  notes that, for the functions studied in Theorems~\ref{theo:exp} 
  and~\ref{theo:unisingularparabolic}, we recover the family
  of maps $g_{a,b}$ from~\eqref{eqn:tangent}, up to conformal conjugacy.
  \begin{example}
\label{ex:exp}
Suppose that $f$ has  a  Fatou component $U$ 
 of period $n$ such that $U$ contains only one singular value 
 of $f$, and such that $f^n\colon U\to U$ is of infinite valence. Then 
 this restriction is a universal covering over a single point $a\in U$, and
 therefore Theorem~\ref{theo:finite} applies with $p=0$ and $q=1$. 

 In particular, suppose that
  $f(z) = \lambda e^z$, for some $\lambda \ne 0$, such that
  $f$ has a forward-invariant attracting or parabolic
  Fatou component, $U$. Then $U$ must contain the 
  origin, which is the only singular value of $f$, and
  Corollary~\ref{cor:exponentials} applies with 
  $P \equiv \lambda$, $p=0$, $Q \equiv \operatorname{Id}$ and $q=1$. 
  
The only Blaschke product of order zero is the rotation. Hence there exist 
$\sigma\in (-\pi,\pi]$, $c>0$, and $\theta \in (-\pi, \pi]$ such that
\begin{equation}
\label{e1}
g(z) = e^{i\sigma} \exp\left(c \frac{z+e^{i\theta}}{z-e^{i\theta}}\right).
\end{equation}
Conjugating $g$ with a rotation if necessary, we can assume that $\theta = \pi$, in which case 
\begin{equation}
\label{e2}
g(z) = \exp\left(i\sigma + c \frac{z-1}{z+1}\right).
\end{equation}
This is equivalent to~\eqref{eq:DevandG} and~\eqref{eqn:tangent}, with 
 $\mu = \sigma + ic$ and $(a,b)=(c,\sigma)/2$, respectively.

 In particular, when $U$ is a parabolic basin, and 
  in particular for $f(z)=e^{z-1}$, we have $c=2$ and $\sigma=0$
   by Theorem~\ref{theo:unisingularparabolic}.
   So here $g$ takes the form
\[
g(w) = \exp\left(\frac{2(w-1)}{w+1}\right).
\]
\end{example}
\begin{example}
\label{ex:zexp}
Suppose we are in the family $f(z) = \lambda z e^z$, so that $P \equiv \lambda \operatorname{Id}$, $p=1$, $Q \equiv \operatorname{Id}$ and $q=1$. Suppose that $f$ has a forward-invariant Fatou component, $U$, of infinite valence, that contains the origin. Then the hypotheses of Corollary~\ref{cor:exponentials} hold.

The fact that $q=1$ means, again, that after the right choice of $\arg \phi'(0)$ we can take, for some positive $c$,
\[
S(z) = \exp\left(c \frac{z-1}{z+1}\right).
\]

For $B$, the fact that $p=1$ means that $B$ is a Blaschke product of degree one. However, we also know that $f(0) = 0$ and so $g(0) = 0$. Thus, for some $\sigma>0$, we have
\[
g(z) = e^{i\sigma} z \exp\left(c \frac{z-1}{z+1}\right).
\]
\end{example}
\begin{example}
\label{ex:powerexp}
Suppose we are in the family $f(z) = \lambda e^{z^q}$, for some $q \in \N$, so that $P \equiv \lambda$ and $Q(z) = z^q$. Suppose that the hypotheses of Corollary~\ref{cor:exponentials} hold. Since $f$ omits the origin, we get that the Blaschke product is just a constant. We also get, by obvious symmetry considerations, that the $\theta_j$ in \eqref{eq:Sdef} can be taken to be the $q$-th roots of unity, and the $c_j$ are all the same. Set $\omega = e^{2\pi i/q}$. Then there exists $\sigma>0$ and $c>0$ such that
\[
g(z) = e^{i\sigma} \exp\left(-c\sum_{j=1}^n \left(\frac{\omega^j+z}{\omega^j - z}\right)\right).
\]
\end{example}
%

%
%
%
\section{Proof of Theorem~\ref{theo.sine}}
\label{S.sine}
\begin{proof}[Proof of Theorem~\ref{theo.sine}]
It is easy to see that if $\lambda \in (0, 1)$, then 
 both critical values $\pm \lambda$ are in the immediate basin
 of attraction $U$ of the attracting fixed point at $0$. It follows
 easily that $U=F(f)$. (This was already observed by Fatou~\cite{fatou}.) 
  In particular,~\ref{p1} holds. 

For simplicity write $f$ for $f_\lambda$. Choose the Riemann map $\phi\colon \D \to U$ so that $\phi(0) = 0$ and $\phi'(0) > 0$. Then $\phi$ maps points on the real line to the real line, because of the obvious symmetry of $U$. Let $g$ be the inner function $g \defeq \phi^{-1} \circ f \circ \phi$, and let
\[
g = B \cdot S,
\]
where $B$ is a Blaschke product and $S$ is a singular inner function, as usual. 

Clearly $0$ is a simple zero of $f$, and so of $g$, and so of $B$. Notice that $f$ is $2 \pi$-periodic. Notice also that $f^2$ (the second iterate) is $\pi$-periodic. Hence $J(f)$ is $\pi$-periodic. The zeros of $f$ are the points $\pm\zeta_n$ where $$\zeta_n = n\pi, \qfor n \in \N.$$ Hence the other zeros of $g$, and so of $B$, are all of the form $\pm a_n$, for some increasing sequence with
 $a_n \rightarrow 1$ as $n\rightarrow \infty$.  Thus we can write $B$ as
\begin{equation}\label{eqn:B}
B(z) = z \cdot \prod_{n \geq 1} \frac{a_n^2 - z^2}{1 - a_n^2 z^2}.
\end{equation}

We claim that $S$ is, in fact, absent. To prove this claim, suppose otherwise. 
If $D\subset U$ is a Jordan domain containing $[-\lambda,\lambda]$, then
$f^{-1}(U\setminus D)$ has exactly two components. Indeed, 
 by the elementary mapping properties of $\sin$, the set 
 $f^{-1}(\partial D)$ consists of two curves tending to $\pm\infty$, symmetrically
 with respect to the real axis. 
By Theorem~\ref{theo:tracts}, the map $g$ has two singularities on $\partial \D$. Since $\pm 1$ are singularities of $B$, this means that $S$ has at most two singularities, which would need to be positioned at $\pm 1$. 
Since $U$ is symmetric on reflection in the imaginary axis, and since $f$ is an odd function, our choice of $\phi$ implies that $g$ is also an odd function. Hence $S$ is generated by two equal masses, each at $\pm 1$. In particular, by a calculation
from~\eqref{eq:singinnerdef}, there exists $c>0$ such that
\[
S(z) = \exp \left(c \frac{z^2 + 1}{z^2 - 1}\right).
\]
It follows that as $x \rightarrow 1$, we have that $S(x) \rightarrow 0$, and so $g(x) \rightarrow 0$. This is impossible, as $f(x)$ does not have a limit as $x \rightarrow \infty$. This contradiction proves our claim, and we have 
 $g=B$ with $B$ as in~\eqref{eqn:B}.

Next we seek to find a formula for the $a_n$. Because of the periodicity of $J(f)$ the hyperbolic distance in $U$ from $\zeta_n$ to $\zeta_{n+1}$ is constant, and in fact equal to the hyperbolic distance in $U$ from $0$ to $\zeta_1$. Call this distance $d$. By symmetry, the real axis is a hyperbolic geodesic of $U$. It
follows that $\dist(0,\zeta_n)=n\cdot d$. 

Then the hyperbolic distance in $\D$ from $0$ to $a_n$ is also $n\cdot d$, and so 
\[
\log \frac{1+a_n}{1-a_n} = n\cdot d,
\]
from which we calculate
\[
a_n = \frac{e^{n\cdot d} - 1}{e^{n\cdot d} + 1}= \frac{\tau^{n} - 1}{\tau^{n} + 1}
\]
where $\tau = e^d\in (1,\infty)$.

Clearly $d$, and hence $\tau$, depends on $\lambda$. Write $\tau = \tau(\lambda)$; it remains to show that $\tau\colon (0,1)\to (1,\infty)$ is a homeomorphism. Since $\phi(0) = g(0) = f(0) = 0$, 
\[
\lambda = f'(0) = g'(0) = \prod_{n \geq 1} a_n^2 = \prod_{n \geq 1} \left(\frac{\tau^{n} - 1}{\tau^{n} + 1}\right)^2.
\]
 In particular, $\lambda$ is uniquely determined by $\tau$. 
The function $x\mapsto (x-1)/(1+x)$ is strictly increasing on 
 $[1,\infty)$. So $\lambda$ is a strictly increasing continuous
 function of $\tau$. Moreover, it is easy to see that
  $\lambda\to 0$ as $\tau\to 1$, and $\lambda\to 1$ as $\tau\to\infty$. 
\end{proof}
%

%
%
%
\section{Fatou components with infinitely many critical values}
\label{S.infinite}
\begin{proof}[Proof of Theorem~\ref{theo.Fatou}]
It is easy to show that $f_\lambda$ has a completely invariant Fatou component, $U_\lambda$, which contains a right half-plane. Hence \ref{p1} holds. 

For simplicity write $f$ for $f_\lambda$ and $U$ for $U_\lambda$. Note that $0 \in U$, and indeed (by a calculation) $\R \subset U$. Note also that $f(\overline{z}) = \overline{f(z)}$, and so $U$ is symmetric on reflection on the real line.

Let $\alpha > 0$, and choose the Riemann map $\phi : \HH \to U$ so that $\phi(i\alpha) = 0$ and $i\phi'(i\alpha) > 0$. (We will choose $\alpha$ later). Then $\phi^{-1}$ maps points on the real axis to the positive imaginary axis, because of the symmetry of $U$. Let $h \defeq \phi^{-1} \circ f \circ \phi$ be an inner function of the upper half-plane.

Note that $f(w + 2\pi i) = f(w) + 2\pi i$, for $w \in \C$. This means that $U$ is periodic under translation of $2\pi i$. We can deduce that there exists $\kappa > 0$ such that $$\phi^{-1}(w + 2 \pi i) = \phi^{-1}(w) - \kappa, \qfor w \in \C.$$ It then follows that $h(z - \kappa) = h(z) - \kappa$, for $z \in \HH$.

We claim that $h$ has one singularity, and this is at infinity. 
 (Observe that we cannot apply Theorem~\ref{theo:tracts}, as the 
 singular values of $f$ are not compactly contained in $U$.) Suppose that $\zeta$ is such that $|\zeta|$ is small. It can be shown by a calculation that the preimages under $f$ of $\zeta$ that are of large modulus are close to the points $-\log |y_n| + iy_n$, where 
\[
y_n = 
\begin{cases}
\frac{4n+1}{2}\pi, \text{ for } n \in \N, \\
\frac{4n+3}{2}\pi, \text{ for } -n \in \N.
\end{cases} 
\]
These points can be connected to infinity by two curves in $U$ (one containing the points of positive imaginary part, and the other containing the points of negative imaginary part) that are each homotopic to $(0, +\infty)$. This establishes the fact that $h$ has only one singularity, since by transferring everything to the unit disc via a M\"obius map we can deduce that $g$ has exactly one singularity on $\partial \D$, the point where all preimages of almost every $z \in \D$ accumulate. Moreover, this singularity of $h$ is at $\lim_{x \to +\infty} \phi^{-1}(x) = i\infty$. This completes the proof of the claim.

Since $h$ has no singularities in $\overline{\HH}$, by Schwarz reflection it
 extends to a transcendental meromorphic map of the whole plane, which we continue to call $h$, and which maps $\HH$ to itself. For simplicity we now write
\[
h(z) \defeq z + G(z) \defeq z + \frac{G_1(z)}{G_2(z)},
\]
where $G_1$ and $G_2$ are entire. Then
\[
G(z + \kappa) = G(z),
\]
in other words, $G$ is $\kappa$-periodic. Note that $\kappa$ depends linearly on $\alpha$; this can easily be seen by pre-composing $\phi$ with a map $z \mapsto cz$, for $c > 0$. Hence we can assume that $\alpha$ is chosen so that $\kappa = \pi$; in other words, $G$ is $\pi$-periodic.

Next we locate the poles and fixed points of $h$. We have $\phi(it)\in\R$ and $\phi(it)\to -\infty$ as $t\to 0$; hence $\phi(h(it))=f(\phi(it))\to +\infty$. So $h(it)\to\infty$ as $t\to 0$, and $0$ is a pole of 
$h$. Since $h$ commutes with translation by $\kappa=\pi$, all the integer multiples of $\pi$,  $\zeta_n=\pi n$, are also poles of $h$. We claim that there are no other poles. Indeed, 
 every pole of $h$ is the landing point of some piece of $h^{-1}(i\cdot [\alpha,\infty)) = \phi^{-1}(f^{-1}([0,\infty)))$. But 
 $f^{-1}([0,\infty))$ consists of countably many curves $\gamma_n$, each tending to $-\infty$ asymptotically at imaginary part $2n\pi$. 
  It follows easily that $\phi^{-1}(\gamma_n)$ lands at $\zeta_{-n}$, and these are the only poles of $h$.

Now we locate the fixed points of $h$. The fixed points of $f$ are the points $z_n \defeq -\log\lambda + (2n+1)\pi i$, for $n \in \Z$. These points are accessible boundary points of $U$, since, for each $n \in \N$, the set $\{ z_n + x : x > 0 \}$ lies in $U$. Since the set of fixed points is $2 \pi i$-periodic, the corresponding fixed points $w_n \defeq \phi^{-1}(z_n)$ of $h$ are a $\pi$-periodic set of real numbers; this follows from our choice of $\alpha$ above. Since the fixed points of $f$ are simple, all these are 
simple fixed points of $h$.

It is easy to see that the poles $\zeta_n$ are simple, as otherwise $h$ could not preserve the upper half-plane. Similarly, $h$ cannot have any critical points on the real line. For, if $x \in \R$ and $h'(x) = 0$, then close to $x$ the map $h$ behaves like a power map and so cannot preserve the upper half-plane. 
It follows that there is exactly one fixed point between every two poles. We deduce that the points $z_n$ are the only fixed points of $h$. Moreover, since $z_0$ and $z_{-1}$ are
symmetrically placed with respect to the real axis, the same must be true of $w_0$ and $w_{-1}$. Since $w_{n}-w_{n+1} = \pi$, the $w_n$ are at the \emph{odd} multiples of $\pi/2$.

Since the poles of $h$ are exactly the zeros of the sine function, it follows that there is an entire function $h_2$ such that
\[
G_2(z) = e^{h_2(z)} \sin z.
\]

Similarly, the fixed points of $h$ (which are all simple) are the zeros of $G_1$. Hence there is an entire function $h_1$ such that
\[
G_1(z) = e^{h_1(z)} \cos z.
\]

We have now concluded that there is an entire function $H$ such that
\[
h(z) = z + e^{H(z)} \cot z.
\]
Note that $G$, and hence $e^{H(z)}$, is $\pi$-periodic. 

Now, if $x > 0$ is large, then $f(x + iy) \approx x + iy + \lambda$. We can deduce that if $y > 0$ is large, then $h(x + iy) = x + iy + i\nu(x, y)$, where $\nu(x, y)$ is small. It follows that $e^{H(z)}$ is bounded, and so must be constant. Since $f$ maps the real line to itself and maps large values of $x$ close to $x + \lambda$, $h$ maps the positive imaginary axis to itself and, when $y$ is large, maps $iy$ close to $i(y + \nu)$ for some positive $\nu$. Thus $e^{H(z)}$ is the constant $-\nu$, and we have obtained that
\[
h(z) = z - \nu \cot z.
\]

It remains to show that $\nu = \lambda/2$. Fix $\delta > \max(\pi,\nu)$. For $t>\delta$, we can apply Koebe's distortion theorem
 to $\phi$, restricted to the disc $D(it,t)\subset\HH$ centred at $it$, to obtain distortion estimates on the smaller disc $D(it,\delta)$. More precisely, setting
 $r\defeq \delta/t<1$, we have 
 \[ \frac{(1-r)^2}{(1+r)^2} \leq \frac{\lvert \phi(z) - \phi(it)\rvert }{\lvert z-it\rvert }\cdot \frac{1}{\lvert \phi'(it)\rvert} \leq \frac{(1+r)^2}{(1-r)^2} \]
 for all $z\in D(it,\delta)$. For $z=it + \pi$, we have $\phi(z)=\phi(it)+2\pi$, so $\lvert \phi'(it)\rvert \to 2$ as $t\to\infty$. Similarly, for 
  $z= h(it)$ and sufficiently large $t$, we have $z\in D(it,\delta)$ and 
    \[ \frac{\lvert \phi(z) - \phi(it)\rvert}{\lvert z-it\rvert } = \frac{\lvert f(\phi(it)) - \phi(it)\rvert }{\lvert h(it)-it\rvert } \to \frac{\lambda}{\nu}, \]
     so $\phi'(it) \to \lambda/\nu$. We have shown $\nu = \lambda/2$, as required.
\end{proof}

\section{Proof of Theorem \ref{theo:finitecomplicated}}
\label{sec:approximation}
We first give a simple result about uniform convergence of finite Blaschke products in the unit disc, which we use in the proof of Theorem \ref{theo:finitecomplicated}.
\begin{proposition}
\label{prop:conv}
Suppose that $(B_n)_{n \in \N}$ is a sequence of finite Blaschke products of degree $d$, which converge locally uniformly on $\D$ to a finite Blaschke product, $B$, of degree $d$. Then the convergence is, in fact, uniform on $\D$.
\end{proposition}
\begin{proof}
In general we denote the open disc with centre $w \in \C$ and radius $r>0$ by 
\[
D(w, r) \defeq \{ z \in \C : |z-w| < r \}.
\]
Let $\rho \in (0, 1)$ be such that all the zeros of $B$ lie in the disc $D(0, \rho)$. Set 
\[
t \defeq \min \{ |B(z)| : |z| = \rho \} > 0.
\]
It follows from locally uniform convergence that for all sufficiently large values of $n$, we have 
\[
|B(z) - B_n(z)| < t \leq |B(z)|, \qfor |z| = \rho.
\]
It then follows from Rouch\'e's theorem that, for all sufficiently large values of $n$, all the zeros of $B_n$ lie in $D(0, \rho)$. Hence there exists $r \in (0, 1)$ such that all the zeros of all the $B_n$ lie in $D(0, r)$.

By an easy calculation, we can deduce that there exists $r' > 1$ such that, with $D \defeq D(0, r')$, each $B_n$ is analytic in $D$, and the family $(B_n)_{n \in \N}$ is uniformly bounded in $D$. It then follows by the Vitali-Porter theorem, see, for example, \cite{Schiff}, that the $B_n$ converge locally uniformly to $B$ in $D$. The result follows, as $\overline{\D}$ is a compact subset of $D$.
\end{proof}
Now we give the proof of Theorem~\ref{theo:finitecomplicated}.
\begin{proof}[Proof of Theorem~\ref{theo:finitecomplicated}]
Let $(B_n)_{n \in \N}$ be a sequence of finite Blaschke products that is dense in the space of finite Blaschke products (in the topology of uniform convergence). Such a sequence exists, for example, by choosing functions of the form \eqref{eq:Bdef} with $d$ finite and all the variables $\theta, \Re a_1, \Re a_2, \ldots, \Re a_d$, and $\Im a_1 \Im a_2, \ldots, \Im a_d$ rational.

For each $n \in \N$ let $T_n$ be the translation $T_n(z) \defeq z +4n$, and let $D_n$ be the disc $D_n \defeq D(4n, 1)$. It follows by \cite[Theorem 5.3]{classifyingwandering} that there exists a transcendental entire function $f$ having an orbit of bounded,
simply-connected, escaping, wandering domains $(U_n)_{n \in \N}$ such that the following all hold.
\begin{enumerate}[(A)]
\item $\Delta_n' \defeq D(4n, r_n) \subset U_n \subset D(4n,R_n) \defeq \Delta_n,$ where $0 < r_n < 1 < R_n$ and $r_n,R_n \to 1$ as $n\to \infty$.\label{thesets}
\item $f_{n} \defeq T_{n+1}\circ B_n \circ T_n^{-1}$ is analytic on $\overline{\Delta_n}$, and $|f(z) - f_{n}(z)| \to 0$ as $n \to \infty$
uniformly on $\overline{\Delta_n}$; by ``uniformly'' we mean that for each $\epsilon > 0$ there exists $N \in \N$ such that $|f(z) - f_n(z)| < \epsilon$, for $z \in \Delta_n$ and $n \geq N$.\label{convs}
\item $f:U_n \to U_{n+1}$ has the same degree as $B_n$.\label{degree}
\end{enumerate}

This completes the definition of the function $f$. It remains to show that the Fatou components of $f$ have dynamically associated inner functions with the claimed properties. Suppose that $n \in \N$. Let $\phi_n:\mathbb{D} \to U_n$ be the Riemann map such that $\phi_n(0) = 4n$ and $\phi_n'(0) > 0$. Then 
\begin{equation}
\label{eq:g}
g_{n}= \phi_{n+1}^{-1} \circ f \circ \phi_n
\end{equation}
is an inner function dynamically associated to $f|_{U_n}$. 

We need to be able to approximate the Riemann maps in \eqref{eq:g}, and we claim that $\phi_n - T_n \to 0$ locally uniformly on $\D$ as $n \to \infty$. To prove this claim, we first consider translated copies of $U_n$, defined by
\[
U_n' \defeq T_n^{-1}(U_n), \qfor n \in \N.
\]
Note that, for each $n\in\N$, we have that $D(0, r_n) \subset U_n' \subset D(0, R_n)$. Suppose that $w_0 \in \D$. Then there exists a neighbourhood of $w_0$ that is contained in $U_n'$ for all sufficiently large $n \in \mathbb{N}$. Also, suppose that $w \in \partial \D$. Then it follows by \eqref{thesets} that there exists a sequence of points $w_n \in \partial U_n'$ such that $w_n \to w$ as $n \to \infty$. Hence $U_n' \to \D$ in the sense of kernel convergence; see \cite[p.13]{Pommerenke}. Set
\[
\phi_n' \defeq T_n^{-1} \circ \phi_n, \qfor n \in \N.
\]
It then follows from the Carath\'eodory kernel theorem (\cite[Theorem 1.8]{Pommerenke}) that $\phi_n' \to \operatorname{Id}$ locally uniformly in the unit disc as $n \to \infty$, where $\operatorname{Id}(z) \defeq z$. The claim above follows.

Suppose that $B$ is a given Blaschke product, and suppose that $B$ has degree $d$. Let $(n_p)_{p \in \N}$ be a sequence of integers such that the subsequence $(B_{n_p})_{p \in \N}$ converges uniformly to $B$ on $\D$ as $p\rightarrow\infty$. We can assume that each $B_{n_p}$ has degree $d$.

Note that it follows by \eqref{degree} that the degree of each $g_{n_p}$ is equal to $d$. Observe that the theorem requires that the sequence of functions $(g_{n_p})_{p \in \N}$ converges uniformly on $\D$ to the function $B$ as $p \rightarrow \infty$. We shall prove first that this sequence converges locally uniformly to $B$. We will then use Proposition~\ref{prop:conv} to deduce that this convergence is in fact uniform.

To prove locally uniform convergence, suppose that $K \subset \D$ is a given compact set. Choose $r \in (0, 1)$ sufficiently close to $1$ that $K \subset \Delta$ where $\Delta \defeq D(0, r)$. It follows by the claim above that 
\begin{equation}
\label{eq:phi}
\tilde{\epsilon}_n(z) \defeq \phi_n(z) - T_n(z), \qfor z \in \D,
\end{equation}
is analytic in $\D$, and that $\sup_{z \in \Delta} |\tilde{\epsilon}_n(z)| \to 0$ as $n \to \infty$. Similarly
\begin{equation}
\label{eq:phiminus1}
\hat{\epsilon}_{n}(z) \defeq \phi_n^{-1}(z) - T_{n}^{-1}(z), \qfor z \in U_n,
\end{equation}
is analytic in $U_n$, and that $\sup_{z \in T_n(\Delta)} |\hat{\epsilon}_n(z)| \to 0$ as $n \to \infty$. Finally, it follows by \eqref{convs} above that, for all sufficiently large $n\in \N$, we have that
\begin{equation}
\label{eq:f}
\epsilon_{n}(z) \defeq f(z) - f_{n}(z), \qfor z \in \Delta_n,
\end{equation}
is analytic in $\Delta_n$, and such that $\sup_{z \in \phi_n(\Delta)} |\epsilon_n(z)| \to 0$ as $n \to \infty$. 

Note that, by \eqref{convs}, \eqref{eq:phi}, and \eqref{eq:f}, if $n \in \N$ is sufficiently large, then
\begin{align}
f(\phi_n(z))    &= f_n(\phi_n(z)) + \epsilon_n(\phi_n(z)) \nonumber \\
						 	  &= T_{n+1}(B_n(T_n^{-1}(T_n(z) + \tilde{\epsilon}_n(z)))) + \epsilon_n(\phi_n(z)) \nonumber \\
                &= B_n(z + \tilde{\epsilon}_n(z)) + 4(n+1) + \epsilon_n(\phi_n(z)), \qfor z \in K.\label{eq1}
\end{align}
Since the sequence $(B_{n_p})_{p\in\N}$ converges uniformly to $B$ in $\D$, as $p$ tends to infinity,~\eqref{eq1} gives 
   \begin{equation} T_{n_p+1}^{-1}(f(\phi_{n_p}(z))) \to B(z) \label{eqn:eq2} \end{equation}
uniformly for $z\in K$. 
In particular, if $\tilde{\Delta} = D(0,\tilde{r})$ is a disc containing $B(K)$, then $f(\phi_{n_p}(K))\subset T_{n_p}(\tilde{\Delta})$ for sufficiently large $p$.

It then follows, by \eqref{eq:g}, \eqref{eq:phiminus1}, and \eqref{eqn:eq2}, that, for $z \in K$,
\begin{align*}
g_{n_p}(z) =&\ \phi_{n_p+1}^{-1}\bigl(f(\phi_{n_p}(z))\bigr)  \\ 
           =&\ T_{n_p+1}^{-1}\bigl (f(\phi_{n_p}(z))\bigr) + \hat{\epsilon}_{n_p+1}\bigl(f(\phi_{n_p}(z))\bigr) \to B(z)
\end{align*}
as $p\to\infty$.

We have, therefore, established that the subsequence $g_{n_p}$ converges locally uniformly to $B$ in $\D$, as $p$ tends to infinity. Uniform convergence then follows by Proposition~\ref{prop:conv}. This completes the proof of Theorem~\ref{theo:finitecomplicated}. For, if $g_{n_p} \rightarrow B$ uniformly on $\D$, then, given $\epsilon > 0$, we can set $U= U_{n_p}$ and $V = U_{n_p+1}$ for a sufficiently large value of $p$. Note that $U$ and $V$ are then successive wandering domains in the orbit of $U_0$.
\end{proof}

%
%
\newcommand{\etalchar}[1]{$^{#1}$}
\newcommand{\noopsort}[1]{}

\end{document}